\documentclass[11pt]{amsart}
\usepackage{amscd, amsmath, amsthm, amssymb, xy, xypic, dsfont}
\usepackage{graphicx}

\newtheorem{theorem}{Theorem}[section]
\newtheorem{lemma}[theorem]{Lemma}
\newtheorem{proposition}[theorem]{Proposition}

\newtheorem{question}[theorem]{Question}

\newtheorem{corollary}[theorem]{Corollary}

\theoremstyle{definition}     
\newtheorem{definition}[theorem]{Definition}

\newtheorem{example}[theorem]{Example}

\newtheorem{remark}[theorem]{Remark}
\numberwithin{equation}{section}

\def \hd #1 {\bfseries #1  \mdseries}
\def \italic #1 {\bfseries \it #1 \rm \mdseries}
\def \ra {\rightarrow}
\def \cen #1 { \begin{center} #1 \end{center}}

\def \mbz {\mathbb Z}

\def \mbr {\mathbb R}
\def \mbc {\mathbb C}
\def \mbp {\mathbb P}

\def \mbq  {\mathbb {Q}}
\def \mco  {\mathcal {O}}

\def \mcX {\mathcal {X}}
\def \mcY {\mathcal {Y}}
\def \mcZ {\mathcal {Z}}

\def \Pic {{\rm{Pic}}}
\def \rk {{\,\,\rm{rk}}}
\def \im {{\rm{im}}}

\def \vol {{\rm{vol}}}

\def\K3{\mathcal{K}}

\begin{document}
\title [ Mirror pairs of Calabi--Yau threefolds]
 {Mirror pairs of Calabi--Yau threefolds from mirror pairs of quasi-Fano threefolds}

\author{Nam-Hoon Lee}
\address{
Department of Mathematics Education, Hongik University
42-1, Sangsu-Dong, Mapo-Gu, Seoul 121-791, Korea}
\address{
School of Mathematics, Korea Institute for Advanced Study, Dongdaemun-gu, Seoul 130-722, South Korea }
\email{nhlee@kias.re.kr}
\begin{abstract}
We present a new construction of  mirror pairs of Calabi--Yau manifolds by smoothing normal crossing varieties, consisting of two quasi-Fano manifolds. We introduce a notion of mirror pairs of quasi-Fano manifolds with anticanonical Calabi--Yau fibrations using recent conjectures about Landau--Ginzburg models.
Utilizing this notion, we give  pairs of normal crossing varieties and show that the pairs of smoothed Calabi--Yau manifolds satisfy the Hodge number relations of mirror symmetry.
We consider quasi-Fano threefolds that are some blow-ups of Gorenstein toric Fano threefolds and build 6518 mirror pairs of Calabi--Yau threefolds, including 79 self-mirrors.
\end{abstract}
\maketitle
\section{Introduction}
A \emph{Calabi--Yau manifold} is a compact K\"ahler manifold with trivial canonical
class such that the intermediate cohomologies of its structure sheaf are
all trivial ($h^i(M,\mco_M) = 0$ for $0 < i < \dim(M)$). A $K3$ surface is a Calabi--Yau twofold in this definition.
Calabi--Yau manifolds have special places in the classification of algebraic varieties and they are also among important manifolds that have special holonomy.

To the eye-opening surprise of mathematicians,
Calabi--Yau threefolds happen to be compact factors of the spacetime on which physicists are building their physics theory. They  have been investigating Calabi--Yau threefolds in their own way. They found that different Calabi--Yau threefolds may give rise to the same physics. Those manifolds are called mirror manifolds and the relationship between them is called mirror symmetry. A mirror pair $(M, M^\circ)$ of Calabi--Yau threefolds  is supposed to satisfy
\begin{align}\label{mirrorrel}
 h^{1,1}(M) = h^{1,2}(M^\circ), \,\, h^{1,2}(M) = h^{1,1}(M^\circ).
\end{align}
One can say that the mirror symmetry for Calabi--Yau threefolds from physics has   impacts in the geometry of Calabi--Yau threefolds as follows.
\begin{itemize}
\item Seemingly different two Calabi--Yau threefolds may be  deeply related.
An enumerative problem on one can be translated into another enumerative problem on another which sometimes  is simpler than the original one (\cite{Gi, LLY}).
\item The mirror symmetry expects that Calabi--Yau threefolds exist as pairs. Nowadays it is not an unreasonable question to ask  what  a mirror partner for a certain Calabi--Yau threefold is.
\end{itemize}

Physicists constructed many Calabi--Yau threefolds as hypersurfaces in weight projective spaces, which generate an almost symmetric plot of \mbox{$h^{1,1}-h^{1,2}$} vs.\ \mbox{$h^{1,1}+h^{1,2}$} (\cite{CaLySc, KlSh}).
 V.\ Batyrev generalized the construction and gave completely symmetric mirror construction of Calabi--Yau threefolds as hypersurfaces in Gorenstein toric Fano fourfolds, using the polar duality of reflexive 4-polytopes and proving the Hodge number relation (\ref{mirrorrel}) (\cite{Ba}). This construction was generalized further (\cite{BaBo, Bor}) and  has inspired many researches from both of mathematics and physics.

In this paper, we suggest another systematic construction of mirror pairs of Calabi--Yau threefolds by using the smoothing method.
By smoothing, we mean the reverse process of the semistable degeneration of a manifold to a normal crossing variety.
If a normal crossing variety is the central fiber of a semistable degeneration
of Calabi--Yau manifolds, it can be regarded as a member in a deformation
family of those Calabi--Yau manifolds.
A remarkable difference
between two-dimensional cases of $K3$ surfaces and higher dimensional
cases is that there are multiple deformation types for higher dimensional
Calabi--Yau manifolds. So building a normal crossing variety smoothable to
a Calabi--Yau manifold can be regarded as building a deformation type of
Calabi--Yau manifolds. The construction by smoothing is intrinsically up to deformation.

We consider the simplest case of smoothing where the normal crossing variety is composed of two manifolds. Those two component manifolds will be called quasi-Fano manifolds.
We further assume that the anticanonical linear systems of those quasi-Fano manifolds induce fibrations whose generic fibers are Calabi--Yau manifolds of codimension one.
Even in this simplest case, building the mirror pairs of the smoothing of a normal crossing variety is a challenging problem.
A.\ Tyurin only vaguely suggested that the mirror pair of a smoothing should come from the Landau--Ginzburg models of  components of the normal crossing  variety in the very last part of his posthumous paper (\cite{Ty}).

Shortly after mirror symmetry was  formulated as a duality between Calabi--Yau manifolds, it was
 suggested that Fano manifolds
also may exhibit mirror symmetry.
In this case,  the mirror of a  Fano manifold is not a compact manifold,
but rather a Landau--Ginzburg model, a non-compact manifold equipped  with a regular function called superpotential.
Recently up comes  an interesting conjecture, claiming that the mirror of the Calabi--Yau  smoothing of a normal crossing variety may be topologically obtained by gluing Landau--Ginzburg models of components of the normal crossing variety (\cite{DoHaTh}).
Also in \cite{KaKoPa}, a conjectural Hodge number relation between a variety and its Landau--Ginzburg model has been suggested.
With these as hints, we  try to construct mirror partners of  smoothings of normal crossing varieties.
Our key idea is the realization that we may regard a quasi-Fano manifold with anticanonical fibration as a compactification of a Landau--Ginzburg model of another quasi-Fano manifold.

Roughly speaking, we regard a pair of quasi-Fano manifolds with anticanonical Calabi--Yau fibrations as a mirror pair, if each of them is a compactification of the Landau--Ginzburg model of the other.
After establishing this notion, we define mirror pairs of normal crossing varieties, smoothable to Calabi--Yau threefolds and show that those Calabi--Yau threefolds satisfy the relation (\ref{mirrorrel}).
It turns out that there is a deep connection between mirror symmetry of quasi-Fano threefolds and mirror symmetry of $K3$ surfaces.
We consider quasi-Fano threefolds that are some blow-ups of Gorenstein toric Fano threefolds and build 6518 mirror pairs of Calabi--Yau threefolds, including 79 self-mirrors by smoothing.
One can find  tables for them  in \cite{Lee2}.

The structure of this paper is as follows.

Section \ref{sec2-1} is a background section for the smoothing method. We introduce basic definitions of quasi-Fano manifolds and the smoothing theorem of
Kawamata--Namikawa, which is the main tool of the construction of Calabi--Yau
manifolds in this paper.

We start  Section \ref{sec2-2} by recalling some basic notions and definitions about reflexive polytopes. We  construct a Calabi--Yau threefold  from a quasi-Fano threefold  that is a blow-up of a Gorenstein toric Fano threefold.  This Gorenstein toric Fano threefold comes from a reflexive 3-polytope.

In Section \ref{sec3-1}, we start our journey to find a  mirror partner of the Calabi--Yau threefold constructed in Section \ref{sec2-2} under the guidance of  the conjectures regarding Landau--Ginzburg models. We  make some elementary observations about  quasi-Fano manifolds and their Landau--Ginzburg models.

In Section \ref{sec3-2}, using hints from Section \ref{sec3-1}, we  build a Calabi--Yau threefold  by smoothing a normal crossing variety. The components of  the normal crossing variety are obtained   by  sequentially blowing up another Gorenstein toric Fano threefold that comes from the polar dual of the previous reflexive 3-polytope.
We prove that these Calabi--Yau threefolds   satisfy the relation (\ref{mirrorrel}).
Each of the equivalence classes of reflexive 3-polytopes gives a mirror pair of Calabi--Yau threefolds. Hence we  obtain a big list of mirror pairs of Calabi--Yau threefolds, which are summarized in Table 1 of \cite{Lee2}.

In Section \ref{sec4}, motivated by our success in the previous sections, we  define a notion of mirror pairs of quasi-Fano threefolds. This definition utilizes the conjectural Hodge number relations in \cite{KaKoPa} and the mirror symmetry of $K3$ surfaces in \cite{Do}.
Then each of equivalence classes of reflexive 3-polytopes gives a mirror pair of quasi-Fano threefolds, which are also listed in Table 1 of \cite{Lee2}. Those will be building blocks for the construction of more mirror pairs of Calabi--Yau threefolds in Section \ref{sec5}.
 We also introduce  other examples of mirror pairs of quasi-Fano threefolds which come from non-symplectic involutions on $K3$ surfaces.

We start Section \ref{sec5} by introducing a notion of mirror pairs of normal crossing varieties smoothable to Calabi--Yau threefolds. We  prove Theorem \ref{3dimmirrort}, claiming that the expected relation (\ref{mirrorrel})  hold for the smoothing of those pairs. Combining mirror pairs of quasi-Fano threefolds we obtained before, we give another large table of mirror pairs of Calabi--Yau threefolds, including 79 self-mirrors.
Those are listed in Table 2 of \cite{Lee2}.

Section \ref{sec6} is  about  whether our pairs of Calabi--Yau threefolds are new.  We pick up a particular example of Calabi--Yau threefolds constructed in the previous sections and show that it is not homeomorphic to any of the Calabi--Yau threefolds that are desingularizations of anticanonical sections of  Gorenstein toric Fano fourfolds (\cite{Ba, KrSk2}).

 Section \ref{sec7} is devoted to some discussion on the higher dimensional generalization of notions from the previous sections.
We suggest a definition of a mirror pair of higher dimensional quasi-Fano manifolds and prove a topological mirror relation.

In Section \ref{sec8}, we discuss quasi-Fano manifolds with anticanonical fibrations that do not have quasi-Fano manifolds as their mirror partners.

Rigorously speaking, the pairs of Calabi--Yau manifolds constructed in this paper are  conjectural  because we only check the Hodge number relations of mirror symmetry.
However, seeing that those ingredients from Landau--Ginzburg models, mirror symmetry of $K3$ surfaces and mirror symmetry of Calabi--Yau manifolds are merged very naturally to produce expected results, we expect  that they are genuine mirror pairs.

\section{Preliminaries}\label{sec2-1}

 By a variety, we mean a reduced complex analytic space.
We start with defining basic terminologies.
\begin{definition}
A \emph{quasi-Fano manifold} $X$ is a smooth projective variety whose anticanonical linear system $|-K_X|$ contains a Calabi--Yau manifold and
$$h^i(X, \mco_X) =0$$
for $i>0$.
\end{definition}
We denoted the Calabi--Yau manifold by $D_X$. If a generic element of $|-K_X|$ is smooth, then $D_X$ will be referred to one of those generic ones.

Let
$$\Pic_X(D_X) = i^*(\Pic(X)) \subset \Pic(D_X)$$
and $\alpha_X = \rk \Pic_X(D_X)$ be the rank of the group  $\Pic_X(D_X)$,
where $i:D_X \hookrightarrow X$ is the inclusion map.
Note that $\Pic_X(D_X)$ is a subgroup of  $H^2(D_X, \mbz)$.
If the normal bundle $N_{D_X/X}$ to $D_X$ in $X$ is trivial, then the anticanonical linear system $|D_X|$ induces a fibration (to be called  \emph{anticanonical fibration})
\begin{align}
 \overline W_X:X \ra \mbp^1 \label{refree1}
\end{align}
with $\overline W_X^{-1}(\infty) =D_X$ and $X$ is said to have an \emph{anticanonical Calabi--Yau fibration}.

Let $\mathcal X =X_1 \cup X_2$ be a variety whose irreducible components are two smooth varieties $X_1$ and $X_2$.
$\mathcal X$ is called a \emph{normal crossing variety} if, near any point $p \in X_1 \cap X_2$, $\mathcal X$ is locally isomorphic to
$$\{(x_0,x_1,\cdots, x_n) \in \mbc^{n+1} | x_{n-1} x_n = 0 \}$$
 with $p$ corresponding to the origin and $X_1$, $X_2$ locally corresponding to the hypersurfaces $x_{n-1}=0,  x_n = 0$ respectively in $\mbc^{n+1}$.
 Note that the variety $D_\mcX:=X_1 \cap X_2$ is  smooth.
Suppose that there is  a proper map  $\psi : \mathfrak{X} \ra B$ from a K\"ahler
manifold ${\mathfrak X}$ onto the unit disk $B=\{t \in \mbc | \| t\| \leq 1 \}$ such that
the fiber ${\mathfrak X}_t = \psi^{-1}(t)$ is a smooth manifold for every $t \neq 0$ and ${\mathfrak X}_0 = \mathcal X$.
We say that $\mathcal X$ is a semistable degeneration of a smooth threefold $M={\mathfrak X}_t$ ( $t \neq 0$) and that $M$ is a \emph{semistable smoothing} (simply smoothing) of $\mcX$ ( \cite{KaNa}).

Consider a normal crossing variety $\mathcal X =X_1 \cup X_2$ of quasi-Fano manifolds $X_1, X_2$ such that $D_\mcX = X_1 \cap X_2$ is an anticanonical section of both $X_1, X_2$.
If $\mathcal X$ is projective and the bundle
$$N_{D_\mcX /X_1} \otimes N_{D_\mcX /X_2} $$
 on $D_\mcX$ is trivial (called $d$-semistability), then $\mcX$ is smoothable to a Calabi--Yau manifold (will be denoted by $M_\mcX$) (Theorem 4.2, \cite{KaNa}).

 For dimension three, the Hodge numbers of $M_\mcX$ are given by (Corollary 8.2, \cite{Lee}):

 \begin{equation}
\label{hodgenumberq}
 \left \{\, \begin{aligned}
       h^{1, 1} (M_\mcX) &= h^{1,1} (X_1) + h^{1,1} (X_2) - \alpha_\mcX -1,\\
         h^{1, 2} (M_\mcX)&= h^{2,1} (M_\mcX)= 21 + h^{1, 2} (X_1) + h^{1, 2} (X_2) - \alpha_\mcX,
       \end{aligned}
 \right.
\end{equation}
where
$$\alpha_\mcX = \rk \left ( \Pic_{X_1}(D_\mcX) + \Pic_{X_2}(D_\mcX) \right).$$
Let $X_i$ be a quasi-Fano manifold with a smooth anticanonical section $D_{X_i}$ for each $i=1, 2$. If $D_{X_1}$ and $D_{X_2}$ are isomorphic, one can make a $d$-semistable normal crossing variety  $X_1 \cup_D X_2$  by gluing along $D_{X_1}$ and $D_{X_2}$, where `$\cup_D$' means gluing $X_1, X_2$ along $D_{X_1}$, $D_{X_2}$ (see \S 2 (Corollary 2.4) of \cite{KeTa} for details for the gluing process).

Let $X$ be a quasi-Fano manifold such that  the normal bundle $N_{D_X/X} $ to $D_X$ in $X$ is trivial. Then we have  the anticanonical fibration  (\ref{refree1})
$$ \overline W_X:X \ra \mbp^1$$
with $\overline W_X^{-1}(\infty) =D_X$.
 Let $X_1, X_2$ be copies of $X$. We denote the copy in $X_i$ of $D_X$ by $D_{X_i}$.
 We make a normal crossing variety $\mcX = X_1 \cup_D X_2$ (also to be denoted by $X \cup_D X$).
 It is easy to see that $\mcX$ is projective and $d$-semistable. Hence, by Theorem 4.2 in \cite{KaNa}, it is smoothable to a Calabi--Yau manifold $M_\mcX$. Since the manifold $M_\mcX$ is determined by $X$ up to deformation, we will denote it by $\Xi_X$.
 In fact, one can construct $\Xi_X$ as a branched double cover over $X$, branched along $S=\overline W_X^{-1}(\{0, \infty\})$ but we will keep this point of view of smoothing for the time being and  we will come back to this point later in Section  \ref{sec7}.

\section{The Calabi--Yau threefold $\Xi_{X_\Delta}$} \label{sec2-2}

We recall some notations from toric geometry.
An integral polytope $\Delta$ in $\mbr^n$ is a convex hull of finitely many integral points (points with integer coordinates).
If, for integral polytopes $\Delta_1, \Delta_2$, there is a $\mbz^n$-preserving
affine transformation $\sigma$ satisfying $\Delta_2 =  \sigma(\Delta_1)$,  then $\Delta_1$, $\Delta_2$ are said to be equivalent.
For a set $A \subset \mbr^n$, its polar dual $A^\circ$ is defined by
$$A^\circ=\{u \in \mbr^n |  u \cdot v \leq -1 \text{ for any } v \in A  \},$$
where `$.$' is the standard inner product in $\mbr^n$.
An integral polytope $\Delta \subset \mbr^n$ is called a reflexive $n$-polytope if $(0,\cdots, 0)$
is in the interior of $\Delta$ and its polar dual $\Delta^\circ$ is also a lattice polytope.
For a face $\Gamma$ of $\Delta$, $l(\Gamma)$ and $l^*(\Gamma)$  are the numbers of integral points in $\Gamma$ and   in the relative interior of $\Gamma$ respectively.
We let $\Delta[k]$ be the set of $k$-dimensional faces of $\Delta$.
For each $\Gamma \in \Delta[k]$, let $\Gamma^\circ$ be the dual $(n-k-1)$-dimensional face of $\Delta^\circ$.

For a fan $\Sigma$ in $\mbr^n$, we denote by $X(\Sigma)$ the associated toric variety and let $\Sigma[1]$ be the set of primitive ray generators of
 $\Sigma$. Hence $\Sigma[1]$ is a set of integral points.

For a reflexive polytope $\Delta$, we denote by $\mbp(\Delta)$ the toric variety  that is associated with the fan consisting of cones over all the proper faces over $\Delta$ --- \emph{this is different from notations in \cite{Ba}}. It is known that $\mbp(\Delta)$ is a Gorenstein toric Fano variety.
Fix a fan $\Sigma_\Delta$, consisting of cones over simplices in $\partial \Delta$ in a maximal
coherent triangulation of $\Delta$.  A maximal
coherent triangulation is defined and proved to exist in \cite{GeKaZe}.
Note $\Sigma_\Delta[1] = \partial \Delta \cap \mbz^n$.
 The toric variety $X(\Sigma_\Delta)$, which is projective, is called a maximal partial projective crepant  desingularization of $\mbp({\Delta})$ (\cite{Ba}).

For a fan $\Sigma$, a reflexive polytope $\Delta$ and a quasi-Fano manifold $X$,
we summarize our notations, including ones to be defined:
\begin{itemize}
\item $X(\Sigma)$ is the associated toric variety of $\Sigma$.
\item $\Sigma[1]$  is the set of primitive ray generators of
the fan $\Sigma$.
\item $\Delta[k]$  is the set of $k$-dimensional faces of $\Delta$.
\item For a face $\Gamma$ of a polytope, $l(\Gamma) = | \Gamma \cap \mbz^n|$ is the number of integral points in $\Gamma$.
\item  $l^*(\Gamma)$ is the number of integral points in the relative interior of $\Gamma$.
\item $\mbp(\Delta)$ is the Gorenstein toric Fano variety that is associated with the fan consisting of cones over all the proper faces over $\Delta$.
\item $\Sigma_\Delta$ is the fan consisting of cones over a maximal projective triangulation of $\Delta$. $\Sigma_\Delta[1] = \partial \Delta \cap \mbz^n$.
\item $X(\Sigma_\Delta)$ is the toric variety associated with the fan $\Sigma_\Delta$. It is
 a maximal partial projective crepant  desingularization of $\mbp({\Delta})$.
\item $D_X$  is a (generic) smooth anticanonical section of $X$.
\item $\Pic_X(D_X) = i^* ( \Pic(X))$, where $i: D_X \hookrightarrow X$ is the inclusion.
\item $\alpha_X = \rk \Pic_X(D_X)$.
\item $X_\Delta$ is the blow-up of $X(\Sigma_{\Delta})$ along a smooth curve $c \in \left |-K_{X(\Sigma_{\Delta})}|_{D_{X(\Sigma_{\Delta})}} \right |$.
\item $Y_\Delta$ is a sequential blow-up of $X(\Sigma_{\Delta})$ along the smooth irreducible curves $c_1, c_2, \cdots, \cdots, c_k$ (p.\ \pageref{ydelta}).
\item $X_1 \cup_D X_2$ is the normal crossing variety of quasi-Fano manifolds $X_1, X_2$, made by gluing along their isomorphic smooth anticanonical sections.
\item $M_{\mathcal X}$ is a smoothing of a normal crossing variety $\mathcal X = X_1 \cup X_2$.
\item $X\cup_D X$ is a normal crossing variety, made by gluing two copies of $X$ along the two copies of $D_X$.
\item $\Xi_X = M_{\mcX}$ for $\mcX= X \cup_D X$.
\item $ W_X:X^* \ra \mbc$ is a Landau--Ginzburg model from an anticanonical fibration $\overline W_X : X \ra \mbp^1$, where $X^* = X - D_X$, $ W_X = \overline W_X|_{X^*}$  (p.\ \pageref{lgm}).
\item $\mcZ_\Delta = X_\Delta \cup_D Y_\Delta$ (p.\ \pageref{zdelta}).
\end{itemize}

For dimension three,  there are 4319 equivalence classes of reflexive 3-polytopes (\cite{KrSk}) and  a
maximal partial projective crepant desingularization $X(\Sigma_\Delta )$ of $\mbp(\Delta)$ is always smooth (\cite{Ba}).
The following lemma will be used several times.
\begin{lemma} \label{reflexeqn} For a reflexive 3-polytope $\Delta$,
$$3! \vol(\Delta)=2l(\Delta)-6$$
and
$$l(\Delta) +  l({\Delta^\circ}) + \sum_{\theta \in \Delta[1] } l^*(\theta) l^*(\theta^\circ) -  \sum_{v \in \Delta[0]} l^*(v^\circ) -\sum_{\Gamma \in \Delta[2]} l^*(\Gamma) = 28.$$
\end{lemma}
\begin{proof}
The first equality is well-known with the following basic property of reflexive 3-polytope (see, for example, \cite{Ka}):
$$\sum_{\theta \in \Delta[1] } (l^*(\theta)+1) (l^*(\theta^\circ)+1) =24.$$
Combining this with
$$l(\Delta) = \sum_{\Gamma \in \Delta[2]} l^*(\Gamma) + \sum_{\theta \in \Delta[1]} l^*(\theta) + \left |\Delta[0] \right |+1,$$
\begin{align}
\label{eqn3} l(\Delta^\circ) = \sum_{v \in \Delta[0]} l^*(v^\circ) + \sum_{\theta \in \Delta[1]} l^*(\theta^\circ) + \left |\Delta[2] \right |+1
\end{align}
and the Euler formula
$$ \left|\Delta[0] \right | - \left | \Delta[1] \right |+ \left | \Delta[2] \right |=2,$$
we have the second equality.
\end{proof}

One can choose a  smooth $K3$ surface $D_{X(\Sigma_\Delta)}$ from the linear system \mbox{$|-K_{X(\Sigma_\Delta)}|$} (Corollary 4.2.3 of \cite{Ba}).
Note that the line bundle $-K_{X({\Sigma_\Delta})}$ is semi-ample (Lemma 4.1.2 in \cite{CoKa}). Hence  the line bundle $-K_{X({\Sigma_\Delta})}|_{D_{X(\Sigma_\Delta)}}$ is nef and
$$\left(-K_{X({\Sigma_\Delta})}|_{D_{X(\Sigma_\Delta)}}\right)^2 = (-K_{X({\Sigma_\Delta})})^3 =  \left( - K_{\mbp(\Delta)} \right )^3 \ge 4.$$
Accordingly, $-K_{X({\Sigma_\Delta})}|_{D_{X(\Sigma_\Delta)}}$ is very ample (Lemma 2.4 of \cite{Kn}) and the linear system $|-K_{X({\Sigma_\Delta})}|_{D_{X(\Sigma_\Delta)}}|$ contains a smooth curve $c$.
Let $X_\Delta \ra X(\Sigma_\Delta )$ be the blow-up along   $c$ and  $D_{X_\Delta}$ be the proper transform of $D_{X(\Sigma_\Delta)}$. Now $N_{D_{X_\Delta} /{X_\Delta}}$ to
$D_{X_\Delta}$ in ${X_\Delta}$ is trivial and hence  $X_\Delta$ is a quasi-Fano threefold with the anticanonical fibration
$$\overline W_{X_\Delta} : X_\Delta \ra \mbp^1$$
and $\overline W_{X_\Delta}^{-1}(\infty) = D_{{X_\Delta}}$.
We obtain a Calabi--Yau threefold $\Xi_{X_\Delta}$ by smoothing $X_\Delta \cup_D X_\Delta$ as in Section \ref{sec2-1}.
From  (\ref{hodgenumberq}), we can calculate
the Hodge numbers of ${\Xi_{X_\Delta}}$ as follows.
\begin{proposition}
$$h^{1,1}({\Xi_{X_\Delta}}) = l(\Delta) +\sum_{\Gamma \in \Delta[2]} l^*(\Gamma )-3$$
and
$$h^{1,2}(\Xi_{X_\Delta}) = 3! \vol(\Delta^\circ)-l( \Delta ) + \sum_{\Gamma \in \Delta[2]} l^*(\Gamma) +27.$$

\end{proposition}

\begin{proof}

By   (\ref{hodgenumberq}), we have
$$h^{1,1}({\Xi_{X_\Delta}}) = 2 h^{2}({X_\Delta}) -1-\alpha_{X_\Delta}$$
and
$$h^{1,2}(\Xi_{X_\Delta}) = 21 +  h^{3} ({X_\Delta})   - \alpha_{X_\Delta}.$$
Firstly,
$$h^{2}(X_\Delta) = h^{2}(X(\Sigma_\Delta)) + 1 = l(\Delta)-3.$$
Recall $\Sigma_\Delta [1] = \partial \Delta \cap \mbz^3$.
 So an integral point $v$ of  $\partial \Delta$ corresponds to a torus invariant divisor $D_v$ of $X(\Sigma_\Delta)$.
 If $v$ lies in a relative interior of a face (codimension one)  of $\Delta$, $D_v$ does not meet with $D_{X_\Delta}$. Let $G$ be a subgroup of $\Pic(X(\Sigma_\Delta))$ that is generated by $D_v$'s, where $v$ does not lie in  a relative interior of a face   of $\Delta$.
 One can show that the map $G \ra \Pic(D_{X({\Sigma_\Delta})})$ is injective (\S 2, \cite{AsGrMo}). Hence
$$\alpha_{X_\Delta} = \alpha_{X(\Sigma_\Delta)} = \rk G =  l( \Delta ) - \sum_{\Gamma \in \Delta [2]} l^*(\Gamma) -4.$$
On the other hand, we have
$$h^3(X_\Delta) = h^3(X(\Sigma_{\Delta})) + h^1(c) = 0+ (-K_{X({\Sigma_{\Delta}})}^3)+2 = (-K_{\mbp(\Delta)}^3)+2 =3! \vol(\Delta^\circ)+2 $$
because $ X({\Sigma_{\Delta}})  \ra \mbp(\Delta)$ is a crepant resolution.

Therefore
\begin{align*}
h^{1,1}(\Xi_{X_\Delta}) &= 2 h^{2}(X_\Delta) -1-\alpha_{X_\Delta} \\
                &= 2(l(\Delta)-3) -1 - \left( l( \Delta )  -\sum_{\Gamma \in \Delta[2]} l^*(\Gamma ) -4 \right ) \\
                &= l(\Delta) +\sum_{\Gamma \in \Delta[2]} l^*(\Gamma )-3
\end{align*}
and
\begin{align*}
h^{1,2}(\Xi_{X_\Delta}) &= 21 +  h^{3} (X_\Delta)   - \alpha_{X_\Delta} \\
 &= 21 + 3! \vol(\Delta^\circ)+2
-\left(l( \Delta ) - \sum_{\Gamma \in \Delta[2]} l^*(\Gamma) -4 \right)\\
        &= 3! \vol(\Delta^\circ)-l( \Delta ) + \sum_{\Gamma \in \Delta[2]} l^*(\Gamma) +27.
\end{align*}
\end{proof}

Our next goal is to find a mirror partner of the Calabi--Yau threefold $\Xi_{X_\Delta}$.
One may apply the procedures in the previous section to the polar dual $\Delta^\circ$ of the reflexive polytope $\Delta$ and construct
a Calabi--Yau threefold $\Xi_{X_{\Delta^\circ}}$.
 But one can check immediately that the relation (\ref{mirrorrel}) does not hold for $\Xi_{X_\Delta}$, $\Xi_{X_{\Delta^\circ}}$. Hence this naive try does not work.
Despite this, one may still suspect that a mirror partner of $\Xi_{X_\Delta}$ is somehow related with the variety $X({\Sigma_{\Delta^\circ}})$.
Our plan is to modify $X({\Sigma_{\Delta^\circ}})$ to some other quasi Fano threefold $Y$ so that the Calabi--Yau threefolds $\Xi_{X_\Delta}$, $\Xi_{Y}$ satisfy the relation (\ref{mirrorrel}).
But then the question would be what kind of modification to $X({\Sigma_{\Delta^\circ}})$ is needed.
In the next section, we will have a discussion on this question.

\section{Conjectures and speculations} \label{sec3-1}
For  a quasi-Fano threefold $X$ with anticanonical fibration, $W_X : X \ra \mbp^1$. We have built a Calabi--Yau threefold $\Xi_X$. Let us assume that there is a quasi-Fano threefold $Y$ with anticanonical fibration such that  $\Xi_Y$ is a mirror partner of $\Xi_X$.
In this section, we discuss how those two quasi-Fano threefolds $X, Y$ should be related.
This section is  speculative and is intended for explaining how the author came up with the relations (\ref{eqna}), (\ref{eqnb}) and (\ref{eqn2}) that serve as hints in constructing the threefold $Y_{\Delta^\circ}$ in Section \ref{sec3-2} so that the Calabi--Yau threefolds   $\Xi_{X_\Delta}$, $\Xi_{Y_{\Delta^\circ}}$ satisfy the relation (\ref{mirrorrel}).

We start with recent conjectures about Landau--Ginzburg models.
A \emph{Landau--Ginzburg model} is a pair $(Z, W)$, where $Z$ is a quasi-projective manifold  and $W:Z \ra \mbc$ is a fibration (called \emph{superpotential} ) whose generic   fiber is a Calabi--Yau manifold of codimension one.

In \cite{DoHaTh} (and also in \cite{Au} for simpler case), an interesting conjecture has been made:

\medskip
\italic{Let $M$ be a Calabi--Yau manifold and suppose that $M$ admits a degeneration to a union $X_1 \cup X_2$ of two quasi-Fano varieties glued along an anticanonical hypersurface. Its mirror partner $M^\circ$ can be constructed topologically by gluing together the Landau--Ginzburg models $(Z_1, W_1)$ and $(Z_2,W_2)$ of $X_1$ and $X_2$} respectively.
\medskip

Note that $\Xi_X$ is a smoothing of a normal crossing variety $X_1 \cup_D X_2$ of $X_1$, $X_2$  which are copies of ${X}$. Note  $D_{X_i}=X_1 \cap X_2$ for each $i$.
We have an anticanonical fibration $\overline W_{X_i}:X_i \ra \mbp^1$ with $\overline W_{X_i}^{-1} =D_{X_i}$.
Let $X_i^* = X_i - D_{X_i}$ and consider the map $ W_{{X_i}}:X_i^* \ra \mbc$,  where  $ W_{X_i} = \overline W_{X_i}|_{X_i^*}$.\label{lgm}
Regarding the above conjecture,  we observe the following:
\begin{enumerate}
\item Topologically $\Xi_X$  can be made by gluing the open ends of $X_1^*, X_2^*$ ( \cite{Ty}).
\item The fibration $W_i:X_i^* \ra \mbc$ can be regarded as a superpotential of a Landau--Ginzburg model.
\end{enumerate}
Following the line of thoughts in the conjecture,  we boldly conjecture that the mirror partner ${(\Xi_{X})}^\circ$ of  $\Xi_{X}$ is a semistable smoothing of a normal crossing variety $Y_1 \cup Y_2$ of quasi-Fano threefolds $Y_1, Y_2$ and that
$(X_1^*, W_{X_1})$ and $(X_2^*, W_{X_2})$ are Landau--Ginzburg models of $Y_1, Y_2$ respectively.

Since $X_1, X_2$ are  the same copies of $X$, we expect that $Y_1, Y_2$ are also copies of a single quasi-Fano manifold $Y$ with a smooth anticanonical section $D_Y$.
For the generality of discussion, let us not restrict the dimension $n = \dim Y$ to be three.
The fact that $Y_1 \cup_D Y_2$  is $d$-semistable implies that the  normal bundle
$N_{D_Y/Y} $  on $D_Y$ is trivial.
So the anticanonical linear system of $Y$ induces a fibration $\overline W_Y : Y \ra \mbp^1$ with $\overline W_Y^{-1}(\infty) = D_Y$. We note that the map
$$W_Y : Y^* \ra \mbc$$
 also can be viewed as a superpotential of a Landau--Ginzburg model, where $Y^* = Y - D_Y$ and $W_Y = \overline W_Y|_{Y^*}$.
Noting that $(\Xi_{X})^\circ = \Xi_Y$ can be topologically made by gluing the open ends of $Y_1^*, Y_2^*$, it is reasonable to conclude that $(Y_1^*, W_{Y_1})$ and $(Y_2^*, W_{Y_2})$ are Landau--Ginzburg models of $X_1, X_2$ respectively.
In sum, we speculate:

\medskip
\emph{If $(X^*, W_{X})$ is a Landau--Ginzburg model of $Y$ and  $(Y^*, W_{Y})$  is a Landau--Ginzburg model of $X$, then $(\Xi_X, \Xi_Y$) is a mirror pair of Calabi--Yau manifolds.}

\medskip
In order to find  the quasi-Fano manifold $Y$, we need more information about it.
A conjectural Hodge number relation, suggested in \cite{KaKoPa}, is relevant to this task.
The authors in \cite{KaKoPa} conjecture that if $(X^*, W_X)$ is a Landau--Ginzburg model of $Y$, then the following holds ((3.1.3), \cite{KaKoPa}):
\begin{align}
\label{koneq1} h^{a+n}(X^*, W_X^{-1}(t)) = \sum_{p-q=a}h^{p,q}(Y),
\end{align}
where $t$ is a generic point in the image of $W_X$ and $n=\dim X$.
If $(Y^*, W_Y)$ is a Landau--Ginzburg model of $X$, it becomes
\begin{align}
\label{koneq2} h^{a+n}(Y^*, W_Y^{-1}(t)) = \sum_{p-q=a}h^{p,q}(X).
\end{align}
These two equations imply relations in topological Euler characteristics
\begin{align*}
\chi(X^*, W_X^{-1}(t)) = (-1)^n \chi(Y), \chi(Y^*, W_Y^{-1}(t)) = (-1)^n \chi(X).
\end{align*}
From the pair $(Y^*, W_Y^{-1}(t))$, we have an exact sequence
$$\cdots \ra H^i(Y^*,  W_Y^{-1}(t))) \ra H^i(Y^*) \ra H^i( W_Y^{-1}(t)))  \ra \cdots,$$
which gives a relation of topological Euler characteristic:
$$\chi(Y^*) = \chi(Y^*,  W_Y^{-1}(t))) + \chi ( W_Y^{-1}(t))).$$
Similarly we have
$$\chi(X^*) = \chi(X^*,  W_X^{-1}(t))) + \chi ( W_X^{-1}(t))).$$
Note also $\chi(Y) = \chi(Y^*) + \chi(D_Y)$, $\chi(X) = \chi(X^*) + \chi(D_X)$.
Combining these equations, we have
\begin{equation}
\label{cyeulerrel}
 \left \{\,\, \begin{aligned}
        \frac{1-(-1)^n}{2} \cdot \left(\chi(X) + \chi(Y)\right)  &=  \chi(D_X) + \chi(D_Y),\\
        \frac{1+(-1)^n}{2} \cdot \left(\chi(X) - \chi(Y)\right)  &=  \chi(D_X) - \chi(D_Y),
       \end{aligned}
 \right.
\end{equation}
which implies
\begin{align}
\label{hypermirr} \chi(D_X) = (-1)^{n-1} \chi(D_Y).
\end{align}
Note $\dim D_X = \dim D_Y = n-1$. So this suggests that $(D_X, D_Y)$ should be a mirror pair of Calabi--Yau manifolds.

Assume $n=3$ for the rest of this section.
In the natural map $H_2(D_Y) {\ra} H_2(Y)$, we have
$$\dim H_2(D_Y) = \dim \ker ( H_2(D_Y) {\ra} H_2(Y)) + \dim \im ( H_2(D_Y) {\ra} H_2(Y)).$$
and, by Poincar\'e duality,
$$\dim \im ( H_2(D_Y) {\ra} H_2(Y)) = \dim \im ( H^2(Y) {\ra} H^2(D_Y)) = \alpha_Y,$$
which implies
$$\dim \ker ( H_2(D) \ra H_2(Y)) = h^2(D_Y) - \alpha_Y.$$

From the pair $(Y, D_Y)$, we have an exact sequence
$$0= H_3(D_Y) \ra H_3(Y) \ra H_3(Y, D_Y) \ra H_2(D_Y) \ra H_2(Y) \ra \cdots,$$
which implies
$$h_3(Y, D_Y) = h_3(Y) +\dim \ker ( H_2(D) \ra H_2(Y)) =h^3(Y) +  h^2(D_Y) - \alpha_Y.$$

From the pair $(Y^*, W_Y^{-1}(t))$, we have an exact sequence
\begin{align*}
0= H^1( W_Y^{-1}(t)) \ra H^{2}(Y^*, W_Y^{-1}(t)) &\ra H^2(Y^*) \ra H^2( W_Y^{-1}(t))\\
\ra H^{3}(Y^*, W_Y^{-1}(t)) &\ra H^3(Y^*) \ra  H^3( W_Y^{-1}(t))=0,
\end{align*}
which implies
$$ h^{2}(Y^*, W_Y^{-1}(t)) +  h^2( W_Y^{-1}(t)) + h^3(Y^*) =  h^2(Y^*) +  h^{3}(Y^*, W_Y^{-1}(t)).$$

By Lefschetz duality, $h_3(Y, D_Y) = h^3(Y^*)$ and Equation (\ref{koneq2}) implies
$$h^{2}(Y^*, W_Y^{-1}(t)) = h^{2,1}(X), h^{3}(Y^*, W_Y^{-1}(t))= 2 + 2h^{1,1}(X).$$

Combining theses, finally we have
\begin{align*}
h^3(Y)+h^2(D) -  \alpha_Y &=2h^{1,2}(Y) -h^2(Y)+h^{1,2}(X) +23.
\end{align*}
i.e.\
\begin{align}
\label{eqna}   \alpha_Y &= h^2(Y)-h^{1,2}(X) -1.
\end{align}

Using the assumption that $(Y^*, W_Y)$ is also a Landau--Ginzburg model of $X$, we have
\begin{align}
\label{eqnb}    \alpha_X &= h^2(X)-h^{1,2}(Y) -1.
\end{align}
 Finally, together with (\ref{cyeulerrel}), we have
\begin{align}
\label {eqn2} \alpha_X +\alpha_Y = 20.
\end{align}

In sum, we conjecture

\medskip
\emph{``The relations (\ref{eqna}),  (\ref{eqnb}) and  (\ref{cyeulerrel})
 hold if $\Xi_X, \Xi_Y$ are mirror pairs of Calabi--Yau threefolds''}
 \medskip

This conjecture will be more concretized to form Definition \ref{FTM}  and Definition \ref{higherdefqm}.

\section{Mirror partner of $\Xi_{X_\Delta}$} \label{sec3-2}

Now let us come back to the problem of finding a mirror partner of $\Xi_{X_\Delta}$.
For our previous $X_{\Delta^\circ}$, we have
$$ \alpha_{X_\Delta} +\alpha_{X_{\Delta^\circ}} \le 20,$$
which may not comply with (\ref{eqn2}) in general and  (\ref{eqna}), (\ref{eqnb}) do  not hold.
Instead one can show
$$\alpha_{X_{\Delta}} + \rk \Pic(D_{X_{\Delta^\circ}}) =20,$$
using $\Pic(D_{X_{\Delta^\circ}}) \simeq \Pic(D_{X({\Sigma_{\Delta^\circ}})})$ and results in \cite{Ro}.

Note that $\Pic_{{X_{\Delta^\circ}}}(D_{X_{\Delta^\circ}})$ is a sublattice of $\Pic(D_{X_{\Delta^\circ}})$. Let us investigate which classes of  $\Pic(D_{X_{\Delta^\circ}})$ are missing in  $\Pic_{{X_{\Delta^\circ}}}(D_{X_{\Delta^\circ}})$.
For $v \in \partial \Delta^\circ \cap \mbz^3$, let $D_v$ be the corresponding torus invariant divisor of $X(\Sigma_{\Delta^\circ})$.
Then the intersection $D_v \cap D_{X\left(\Sigma_{\Delta^\circ}\right)}$ may not be irreducible if  $v$ is a point of the relative interior of an edge $\theta$ of $\Delta^\circ$.
So some classes in $\Pic(D_{X_{\Delta^\circ}})$ that come from components of $D_v \cap D_{X\left(\Sigma_{\Delta^\circ}\right)}$ may not come from classes in $\Pic(X_{\Delta^\circ})$ --- these are the missing classes we are looking for.
With this observation, we will find some  modification to $X(\Sigma_{\Delta^\circ})$ that leads to  a quasi-Fano threefold $Y_{\Delta^\circ}$ such that
\begin{enumerate}
\item $\Pic_{Y_{\Delta^\circ}}(D_{{Y_{\Delta^\circ}}}) \simeq \Pic(D_{X(\Sigma_{\Delta^\circ})} ) \,\,\,(\simeq  \Pic(D_{X_{\Delta^\circ}}))$,
\item The normal bundle $N_{{D_{Y_{\Delta^\circ}}}/ {Y_{\Delta^\circ}}}$ is trivial,
\item The relations (\ref{eqna}),  (\ref{eqnb}) and (\ref{eqn2}) hold for $X_\Delta$, $Y_{\Delta^\circ}$.
\end{enumerate}

If one blows up \emph{sequentially} $X({\Sigma_{\Delta^\circ}})$ along \emph{all} the irreducible curves $c_1, c_2, \cdots, c_k$ such that
\begin{align}\sum_{v \in \partial \Delta^\circ \cap \mbz^3} D_v \cap D_{X({\Sigma_{\Delta^\circ}})}=c_1+ c_2+ \cdots + c_k,\label{ydelta}
\end{align}
then he gets a threefold that turns out to satisfy (1), (2) and (3) in the above properties.
The sequential blow-up is done as follows.
Let $Y^{(1)} \ra X({\Sigma_{\Delta^\circ}})$ be the blow-up along $c_1$ and $D^{(1)}$ be the proper transform of $D_{X(\Sigma_{\Delta^\circ})}$. Since the blow-up center $c_1$ lies on $D_{X(\Sigma_{\Delta^\circ})}$, $D^{(1)}$ is isomorphic to $D_{X(\Sigma_{\Delta^\circ})}$. So  $D^{(1)}$ contains copies of $c_1, c_2, \cdots, c_k$. We denote them by $c^{(1)}_1, c^{(1)}_2, \cdots, c^{(1)}_k$.
Note that $\sum_{v \in \partial \Delta^\circ \cap \mbz^3} D_v $ is an anticanonical divisor of $X({\Sigma_{\Delta^\circ}})$. Hence
the divisor $D^{(1)}|_{D^{(1)}}$ is linearly  equivalent to
$$c^{(1)}_1 + c^{(1)}_2+  \cdots + c^{(1)}_k - c^{(1)}_1 = c^{(1)}_2 + c^{(1)}_3+  \cdots + c^{(1)}_k.$$
We construct $Y^{(2)}, Y^{(3)}, \cdots, Y^{(k)}$ inductively as follows.
Let $Y^{(l+1)} \ra Y^{(l)}$ be the blow-up along $c^{(l)}_{l+1}$ and $D^{(l+1)}$ be the proper transform of $D^{(l)}$. Since the blow-up center $c^{(l)}_{l+1}$ lies on $D^{(l)}$, $D^{(l+1)}$ is isomorphic to $D^{(l)}$. So  $D^{(l+1)}$ contains copies of $c^{(l)}_1, c^{(l)}_2, \cdots, c^{(l)}_k$. We denote them by $c^{(l+1)}_1, c^{(l+1)}_2, \cdots, c^{(l+1)}_k.$
Let $Y_{\Delta^\circ}=Y^{(k)}$.
 Note that
the divisor $D^{(l)}|_{D^{(l)}}$ is linearly  equivalent to
$$ c^{(l)}_{l+1} + c^{(l)}_{l+2}+  \cdots + c^{(l)}_k.$$
Hence the normal bundle $N_{{D_{Y_{\Delta^\circ}}}/ {Y_{\Delta^\circ}}}$ is trivial.
Since the curves $c_1, c_2, \cdots, c_k$ all together generate the lattice   $\Pic(D_{X(\Sigma_{\Delta^\circ})})$ (\cite {Ro}), we have
$$\Pic_{Y_{\Delta^\circ}}(D_{{Y_{\Delta^\circ}}}) \simeq  \Pic(D_{X(\Sigma_{\Delta^\circ})} ).$$
We denote the composite  of the above blow-ups by $Y_{\Delta^\circ} \ra X({\Sigma_{\Delta^\circ}})$.

Note
\begin{align*}
  \alpha_{Y_{\Delta^\circ}}  &=\alpha_{X({\Sigma_{\Delta^\circ}})} + \sum_{\theta \in {\Delta^\circ}[1] } (l(\theta)-2) (l(\theta^\circ)-1-1)\\
             &=  l( \Delta^\circ ) -4 - \sum_{ \Gamma \in {\Delta^\circ}[2]} l^*(\Gamma ) +  \sum_{\theta \in {\Delta^\circ}[1] } (l(\theta)-2) (l(\theta^\circ)-2)\\
           &=  l( \Delta^\circ )  - \sum_{ v \in {\Delta}^{[0]}} l^*(v^\circ ) +  \sum_{\theta \in {\Delta}[1] } l^* (\theta^\circ)l^*(\theta)-4.
\end{align*}
Hence we have
\begin{align*}
  \alpha_{X_{\Delta}} + \alpha_{Y_{\Delta^\circ}} &= \left( l( \Delta ) - \sum_{\Gamma \in \Delta [2]} l^*(\Gamma) -4 \right) + \left(  l( \Delta^\circ )  - \sum_{ v \in {\Delta}^{[0]}} l^*(v^\circ ) +  \sum_{\theta \in {\Delta}[1] } l^* (\theta^\circ)l^*(\theta)-4 \right)\\
   &= 20,
\end{align*}
where Lemma \ref{reflexeqn} was used.
So we explicitly checked that (\ref{eqn2}) is satisfied.

We have
\begin{align*}
k &= | \Delta^\circ [0]| + \sum_{\theta \in {\Delta^\circ}[1] } (l(\theta)-2) (l(\theta^\circ)-1)\\
  &= | \Delta[2]| + \sum_{\theta \in {\Delta}[1] } (l(\theta^\circ)-2) (l(\theta)-1)\\
  &=  | \Delta[2]| + \sum_{\theta \in {\Delta}[1] } l^*(\theta^\circ) l^*(\theta) +\sum_{\theta \in {\Delta}[1] } l^*(\theta^\circ). \\
\end{align*}

 Note
 \begin{align*}
 h^2({Y_{\Delta^\circ}}) &=  h^2(X({\Sigma_{\Delta^\circ}}) + k \\
 &=  l(\Delta^\circ)+  | \Delta[2]| + \sum_{\theta \in {\Delta}[1] } l^*(\theta^\circ) l^*(\theta) +\sum_{\theta \in {\Delta}[1] } l^*(\theta^\circ)-4.\\
 \end{align*}
Let $g(c_i)$ be the genus of $c_i$, then
$$h^3({Y_{\Delta^\circ}}) = h^1(X({\Sigma_{\Delta^\circ}}))+ \sum_i h^1(c_i) =2\sum_i g(c_i) = 2\sum_{v \in {\Delta^\circ}[0]} l^*(v^\circ) =2\sum_{\Gamma \in {\Delta}[2]} l^*(\Gamma).$$
Now one can check that   (\ref{eqna}), (\ref{eqnb}) hold for $X_\Delta$, $Y_{\Delta^\circ}$.

Let us build a Calabi--Yau threefold $\Xi_{Y_{\Delta^\circ}}$ from ${Y_{\Delta^\circ}}$ and calculate the Hodge numbers of $\Xi_{Y_{\Delta^\circ}}$.
Firstly,

\begin{align*}
h^{1,1}(\Xi_{Y_{\Delta^\circ}}) &= 2 h^{2}({Y_{\Delta^\circ}}) -1-\alpha_{Y_{\Delta^\circ}} \\
                &= 2 \left ( l(\Delta^\circ)+  | \Delta[2]| + \sum_{\theta \in {\Delta}[1] } l^*(\theta^\circ) l^*(\theta) +\sum_{\theta \in {\Delta}[1] } l^*(\theta^\circ)-4 \right)-1 \\
                & \,\,\,\,\,\,\,\,\,\,\,\,\,\,\,\,\,\,\,\,\,\,\,\,-\left(l( \Delta^\circ )  - \sum_{ v \in {\Delta}^{[0]}} l^*(v^\circ ) +  \sum_{\theta \in {\Delta}[1] } l^* (\theta^\circ)l^*(\theta)-4\right) \\
                &=  l(\Delta^\circ)+  2 | \Delta[2]| +  \sum_{\theta \in {\Delta}[1] } l^*(\theta^\circ) l^*(\theta) +2 \sum_{\theta \in {\Delta}[1] } l^*(\theta^\circ)+ \sum_{ v \in {\Delta}^{[0]}} l^*(v^\circ )-5 \\
\end{align*}
and

\begin{align*}
h^{1,2}(\Xi_{Y_{\Delta^\circ}}) &= 21 +  h^{3} ({Y_{\Delta^\circ}})   - \alpha_{Y_{\Delta^\circ}}  \\
               &= 21 + 2\sum_{\Gamma \in {\Delta}[2]} l^*(\Gamma)   - \left(l( \Delta^\circ )  - \sum_{ v \in {\Delta}^{[0]}} l^*(v^\circ ) +  \sum_{\theta \in {\Delta}[1] } l^* (\theta^\circ)l^*(\theta)-4 \right) \\
               &= 25 + 2\sum_{\Gamma \in {\Delta}[2]} l^*(\Gamma)   - l( \Delta^\circ )  + \sum_{ v \in {\Delta}^{[0]}} l^*(v^\circ ) -  \sum_{\theta \in {\Delta}[1] } l^* (\theta^\circ)l^*(\theta). \\
\end{align*}

Now we confirm the Hodge number mirror relation (\ref{mirrorrel})  for $\Xi_{X_\Delta}$, $\Xi_{Y_{\Delta^\circ}}$.
\begin{theorem}$\,$ \label{dcmirror}
\cen{$h^{1,1}(\Xi_{X_\Delta}) = h^{1,2}(\Xi_{Y_{\Delta^\circ}})$ and  $h^{1,2}(\Xi_{X_\Delta}) = h^{1,1}(\Xi_{Y_{\Delta^\circ}}).$}
\end{theorem}

\begin{proof}

Firstly,
\begin{align*}
h^{1,1}(\Xi_{X_\Delta}) - h^{1,2}(\Xi_{Y_{\Delta^\circ}}) &= \left (l(\Delta) +\sum_{\Gamma \in \Delta[2]} l^*(\Gamma )-3 \right) \\
&\,\,\,\,\,\,\,\,\,\,\,- \left(25 + 2\sum_{\Gamma \in {\Delta}[2]} l^*(\Gamma)   - l( \Delta^\circ )  + \sum_{ v \in {\Delta}^{[0]}} l^*(v^\circ ) -  \sum_{\theta \in {\Delta}[1] } l^* (\theta^\circ)l^*(\theta)\right)\\
&= l(\Delta) -\sum_{\Gamma \in \Delta[2]} l^*(\Gamma ) + l( \Delta^\circ )  - \sum_{ v \in {\Delta}^{[0]}} l^*(v^\circ ) +  \sum_{\theta \in {\Delta}[1] } l^* (\theta^\circ)l^*(\theta)-28 \\
&=0,
\end{align*}
where Lemma \ref{reflexeqn} was used.

Secondly
\begin{align*}
h^{1,2}(\Xi_{X_\Delta}) - h^{1,1}(\Xi_{Y_{\Delta^\circ}}) &=3! \vol(\Delta^\circ)-l( \Delta ) + \sum_{\Gamma \in \Delta[2]} l^*(\Gamma) +27 \\
&\,\,\,\,-\left(  l(\Delta^\circ)+  2 | \Delta[2]| +  \sum_{\theta \in {\Delta}[1] } l^*(\theta^\circ) l^*(\theta) +2 \sum_{\theta \in {\Delta}[1] } l^*(\theta^\circ)+ \sum_{ v \in {\Delta}^{[0]}} l^*(v^\circ )-5\right)\\
&=l(\Delta^\circ)  -l( \Delta ) + \sum_{\Gamma \in \Delta[2]} l^*(\Gamma) +26 -  2 | \Delta[2]| -  \sum_{\theta \in {\Delta}[1] } l^*(\theta^\circ) l^*(\theta)\\
&\,\,\,\,\,\,\,\,\,\,\,\,\,\,\,\,\,\,\,\,\,\,\,\,\,\, -2 \sum_{\theta \in {\Delta}[1] } l^*(\theta^\circ)- \sum_{ v \in {\Delta}^{[0]}} l^*(v^\circ )\\
 &=-l(\Delta^\circ)  -l( \Delta ) + \sum_{\Gamma \in \Delta[2]} l^*(\Gamma) +28 \\
 &\,\,\,\,\,\,\,\,\,\,\,\,\,\,\,\,\,\,\,\,\,\,\,\,\,\,-  \sum_{\theta \in {\Delta}[1] } l^*(\theta^\circ) l^*(\theta)+  \sum_{v \in \Delta[0]} l^*(v^\circ)
 \,\,\,\,\,\,\,\,\,\,\,\,\,\,\,\,\,\,\,\,\,\,\,\,\,\,\,\,\,\, (\because \,\,\,\, (\ref{eqn3}\,))\\
 &=0,\\
\end{align*}
where Lemma \ref{reflexeqn} was used again.

\end{proof}
Let us take an example.
\begin{example}\label{exam1}                       
Consider a reflexive 3-polytope $\Delta$ whose  vertices are
$$(1,0,0), (0,1,0), (0,0,1), (-4, -4, -3).$$
This reflexive polytope  $\Delta$ gives the quasi-Fano threefold $X_\Delta$ with
$$h^2(X_\Delta) = 6, h^3(X_\Delta) = 38,  \alpha_{X_\Delta} = 4$$
and its polar dual gives another quasi-Fano threefold $Y_{\Delta^\circ}$ with
    $$h^2({Y_{\Delta^\circ}}) = 36,  h^3({Y_{\Delta^\circ}}) = 2, \alpha_{Y_{\Delta^\circ}} = 16.$$
The Hodge numbers of the pair $(\Xi_{X_\Delta}, \Xi_{Y_{\Delta^\circ}})$  of resulted Calabi--Yau threefolds are
  $$ h^{1,1}(\Xi_{X_\Delta}) = 7, h^{1,2}(\Xi_{X_\Delta}) = 55, h^{1,1}(\Xi_{Y_{\Delta^\circ}})=55, h^{1,2}(\Xi_{Y_{\Delta^\circ}})=7.$$
\end{example}

For 4319 equivalence classes of reflexive 3-polytopes, we give all the Hodge numbers of the pairs of $(\Xi_{X_{\Delta}}, \Xi_{Y_{\Delta^\circ}})$'s in Table 1 of \cite{Lee2}.


\section{Mirror pairs of quasi-Fano threefolds}  \label{sec4}
Now we have many examples of $X_{\Delta}, Y_{\Delta^\circ}$ which give pairs of Calabi--Yau threefolds
$(\Xi_{X_{\Delta}}, \Xi_{Y_{\Delta^\circ}})$'s satisfying the relation (\ref{mirrorrel}).
All the pairs $(X_{\Delta}, Y_{\Delta^\circ})$
satisfy
$$\alpha_{X_{\Delta}} + \alpha_{Y_{\Delta^\circ}}=20.$$
It seems that  more delicate structure is involved --- mirror symmetry for lattice polarized $K3$ surfaces.

Let $L$ be a primitive sublattice  of the $K3$ lattice with
 signature $(1, t - 1)$ and $1 \le t \le 19$.
An \emph{$L$-polarized $K3$ surface} is a pair $(S, j)$ where $S$ is a $K3$ surface
and $j : L \ra \Pic(S)$ is a primitive lattice embedding.
One can construct a moduli space ${\K3}_L$, parametrizing $L$-polarized
$K3$ surfaces, which has dimension $20 - t$.
Suppose that there is another lattice $L^\circ$ such that
the orthogonal complement $L^\bot$ in the $K3$ lattice has a decomposition
$$L^\bot  \cong U \oplus L^\circ,$$
where $U$ is the hyperbolic lattice.
The moduli space ${\K3}_{L^\circ}$ is defined as a  mirror of ${\K3}_L$ in \cite{Do} and the pair $(\K3_L,  {\K3}_{L^\circ})$ of moduli spaces was shown to have properties analogous to those of mirror symmetry of Calabi--Yau threefolds.
The lattice pair  $(L, L^\circ)$ is called a \emph{$K3$-mirror pair} of lattices and we have
$$\rk L + \rk L^\circ = 20.$$

In \cite{Ro}, it is noticed  that
$$\left( \Pic_{X(\Sigma_\Delta )} \left(D_{X(\Sigma_\Delta)}\right),  \Pic \left (D_{X(\Sigma_{\Delta^\circ} )}\right) \right) $$ is a $K3$-mirror pair of lattices.
Note
$$\Pic_{X_\Delta} (D_{X_\Delta}) \simeq \Pic_{X(\Sigma_\Delta)} (D_{X(\Sigma_\Delta)}),
\Pic_{Y_{\Delta^\circ}} (D_{Y_{\Delta^\circ}}) \simeq \Pic (D_{X(\Sigma_{\Delta^\circ})}).$$

Hence, for all the  pairs of $(X_\Delta, Y_{\Delta^\circ})$,

\medskip
\emph{ the pair $\left( \Pic_{X_{\Delta}}(D_{X_{\Delta}}), \Pic_{Y_{\Delta^\circ}}(D_{Y_{\Delta^\circ}}) \right)$  of lattices is a $K3$-mirror pair.
}
\medskip

This indicates that there are some  connections between the mirror symmetries for $K3$ surfaces and  pairs of quasi-Fano threefolds that lead to pairs of Calabi--Yau threefolds, satisfying the Hodge number relation (\ref{mirrorrel}).
For a quasi-Fano threefold $X$ whose anticanonical map is a $K3$-fibration, not all $\Pic_{D_X}(D_X)$-polarized $K3$ surfaces appear as an anticanonical section of $X$ --- the anticanonical linear system $|-K_X|$ is just a pencil and so too small to contain all such $K3$ surfaces.
Instead it is expected that a generic $\Pic_{D_X}(D_X)$-polarized $K3$ surface may appear as an anticanonical section of some deformation of $X$.

With all these properties, it seems reasonable to call  $\left( X_{\Delta}, Y_{\Delta^\circ} \right)$ a mirror pair of quasi-Fano threefolds. We give a  definition for three-dimensional case, collecting properties the pairs  $({X_{\Delta}}, {Y_{\Delta^\circ}})$ satisfy.
\begin{definition}\label{FTM}
A pair $(X, Y)$ of quasi-Fano threefolds whose anticanonical maps are $K3$-fibrations is called a mirror pair if
$$\left( \Pic_{X}(D_{X}), \Pic_Y (D_{Y}) \right)$$
 is a $K3$-mirror pair of lattices and
$$ \alpha_X = h^2(X)-h^{1,2}(Y) -1, \alpha_Y = h^2(Y)-h^{1,2}(X) -1,$$
where $D_X$, $D_Y$ are generic smooth anticanonical sections of $X$, $Y$ respectively.
\end{definition}

We show that  a mirror pair of quasi-Fano threefolds give rise a pair of  Calabi--Yau threefolds that satisfy the relation (\ref{mirrorrel}).
\begin{proposition} \label{dcmirrorp}
Let $(X, Y)$ be a mirror pair of quasi-Fano threefolds and consider Calabi--Yau threefolds $\Xi_X$, $\Xi_Y$ from them. Then
$$h^{1,1}(\Xi_X) = h^{1,2}(\Xi_Y),  h^{1,2}(\Xi_X) = h^{1,1}(\Xi_Y).$$
\end{proposition}
\begin{proof}
\begin{align*}
h^{1,1}({\Xi_{X}}) &= 2 h^{2}({X}) -1-\alpha_{X} \\
                   &= 2 ( \alpha_X +h^{1,2}(Y) +1) -1-\alpha_{X} \\
                   &=  \alpha_X +2 h^{1,2}(Y) +1 \\
                   &= 20-\alpha_Y +2 h^{1,2}(Y) +1 \\
                   &=h^{1,2}(\Xi_Y)
\end{align*}
and similarly
\begin{align*}
h^{1,2}(\Xi_{X}) = 21 +  h^{3} ({X})   - \alpha_{X}
                 =   - \alpha_Y +2 h^2(Y) -1
                 =h^{1,1}(\Xi_Y).
\end{align*}
\end{proof}

Let us give more examples of  mirror pairs of quasi-Fano threefolds other than those from toric varieties.
They come from non-symplectic involutions on $K3$ surfaces.

An involution $\rho$ of a $K3$ surface $S$ is called non-symplectic if $\rho^*(\omega) = -\omega $ for each $w \in  H^{2,0}(S)$.
Let $\left( H^2(S, \mbz) \right )^\rho$ be the invariant sublattice of $H^2(S, \mbz)$ by $\rho^*$.
Non-symplectic involutions can be classified by their invariant lattices and there are 75 isomorphic classes of such invariant lattices (\cite{Ni}), which will be called non-symplectic involution lattices.
For a  non-symplectic involution lattice $L$, a generic element of ${\K3}_L$ has a non-symplectic involution whose invariant lattice is $L$.

Fix a  non-symplectic involution lattice $L$ whose orthogonal complement $L^\bot$ in the $K3$ lattice has a decomposition
$$L^\bot \cong U \oplus L^\circ,$$
i.e.\ $(L, L^\circ)$ is a $K3$ mirror pair of lattices. Choose a generic $K3$ surface $S$ from ${\K3}_L$ that has a non-symplectic involution  $\rho$ whose invariant lattice is $L$.
The fixed locus of $\rho$ is a disjoint union of smooth curves.
Let $ \iota: \mbp^1 \ra \mbp^1$ be involution fixing two distinct
points. Let $V_\rho$ be the blow-up of the quotient $(S \times \mbp^1)/ (\rho, \iota)$  along its singular locus.
One can check that  $ V_\rho$ is smooth and it is a quasi-Fano threefold with anticanonical $K3$ fibration with generic fiber $D_{V_\rho}$ isomorphic to $S$ (see \S4 in \cite{KoLe} for the details of the construction). It is also easy to see
$$\Pic_{V_\rho}(D_{V_\rho}) =L.$$

It is known that $L^\circ$ is also a  non-symplectic involution lattice.
For a generic $K$3 surface $S^\circ \in {\K3}_{L^\circ}$ with a non-symplectic involution  $\rho^\circ$ whose invariant lattice is $L^\circ$, construct a quasi-Fano threefold $V_{\rho^\circ}$.
One can easily check that $ ( V_\rho,  V_{\rho^\circ})$ is a mirror pair of quasi-Fano threefolds (See \S4, \cite{KoLe} for the Hodge numbers of them).
It turns out that  the corresponding pair of Calabi--Yau threefolds
$$ \left( \Xi_{ V_\rho}, \Xi_{ V_{\rho^\circ}} \right )$$
is the famous Borcea--Voisin mirror pair of Calabi--Yau threefolds (\cite{Bo, Vo}).  This fact can  be explained as follows.

Choose a point $p \in \mbp^1$ such that $\iota(p) \neq p$.
Consider a degeneration of an elliptic curve $E$  to a normal crossing of two projective lines
$$\mathfrak Z_0 :=\mbp^1 \cup_{p, \iota(p)} \mbp^1,$$
made by attaching at two points $p, \iota(p)$ and denote the degeneration by $\mathfrak Z \ra B$.
 Consider an involution $\tilde \iota$ acting on $\mathfrak Z$ fiberwise which induces  the standard involution $\xi$, acting   on an elliptic fiber as the multiplication by  $-1$ and whose restriction to each component $\mbp^1$  is the involution $\iota$.
Let $\mathfrak X$ be the blow-up of  $(S\times \mathfrak Z ) / (\rho, \tilde \iota) $ along the singular locus. Then $\mathfrak X$ is smooth and the induced map $\mathfrak X \ra B$ is a degeneration of Calabi--Yau threefold $U_{\rho}$ to a $V_\rho \cup_D V_\rho$, where $U_\rho$ is the blow-up of \mbox{$(S \times E)/(\rho, \xi)$} along the singular locus.
Hence we conclude  that $U_{\rho}$ and $\Xi_{V_\rho}$ are of the same deformation type. Similarly
 $U_{\rho^\circ}$ and $\Xi_{V_{\rho^\circ}}$ are of the same deformation type.
 We note that $(U_\rho, U_{\rho^\circ})$ is the mirror pairs of Calabi--Yau threefolds, constructed in \cite{Bo, Vo}.

\begin{remark}
In Definition \ref{FTM}, we imposed the $K3$-mirror lattice condition, which is much stronger than (\ref{eqn2}).  On the other hand, in the proofs of Proposition \ref{dcmirrorp} and  upcoming theorems, only Equation (\ref{eqn2}) is used instead of fully utilizing the $K3$-mirror lattice condition. However we note that every example of mirror pairs of quasi-Fano threefolds, including ones from non-symplectic involutions on $K3$ surfaces, satisfies the $K3$-mirror lattice condition.
We expect that the $K3$-mirror lattice condition will play an important roll in   showing more delicate mirror relations than Hodge numbers relation (\ref{mirrorrel}) between mirror pairs of Calabi--Yau threefolds, constructed by smoothing normal crossings of quasi-Fano threefolds. 
Definition \ref{FTM} needs to be refined, and will be completed when the Landau–Ginzburg models of quasi-Fano theefolds are totally understood.
\end{remark}

\section{Mirror pairs of $d$-semistable Calabi--Yau threefolds of type II}  \label{sec5}
In previous sections, we considered smoothing of normal crossing varieties $X \cup_D X$ whose components are isomorphic. Now we generalize the construction for the case that components are not isomorphic.

Since a normal crossing variety, smoothable to Calabi--Yau manifolds, can be regarded as a member in a deformation
family of those Calabi--Yau manifolds,
 we call the
normal crossing variety
$$X_1 \cup X_2 \cup \cdots \cup X_r$$
 of dimension $n$ as a $d$-semistable Calabi--Yau $n$-fold
of type II if it has no triple locus, i.e.\ $ \bigcup_{i<j<k} X_i \cap X_j \cap X_k = \emptyset$.
This is a generalization of a notion for $K3$ surfaces (\cite{Fr, Ku}).
The normal crossing varieties we are considering are the simplest examples of $d$-semistable  Calabi--Yau manifolds of type II which are composed of two quasi-Fano manifolds and they have been
actively investigated in \cite{Lee, Ty}.
 Now we want to discuss mirror pairs of such Calabi--Yau threefolds.

Consider normal crossing varieties $\mathcal X = X_1 \cup X_2$ and  $\mathcal Y = Y_1 \cup Y_2$ of quasi-Fano threefolds, smoothable to Calabi--Yau threefolds $M_\mcX$ and $M_{\mathcal Y}$ respectively.
Suppose that $(X_i, Y_i)$ is a mirror pair of quasi-Fano threefolds for each $i=1, 2$ and
$M_\mcX$ and $M_{\mathcal Y}$ satisfy the mirror relation (\ref{mirrorrel}). Then one can show:
$$(\alpha_{X_1} +\alpha_{X_2} - \alpha_{\mcX}) + \alpha_{\mcY} =20$$
and
$$ (\alpha_{Y_1} +\alpha_{Y_2} - \alpha_{\mcY}) + \alpha_{\mcX} =20,$$
where
$$\alpha_\mcX = \rk \left ( \Pic_{X_1}(D_\mcX) + \Pic_{X_2}(D_\mcX) \right), \alpha_\mcY = \rk \left ( \Pic_{Y_1}(D_\mcY) + \Pic_{Y_2}(D_\mcY) \right)$$
with $D_\mcX = X_1 \cap X_2$, $D_\mcY = Y_1 \cap Y_2$.
Noting
$$\alpha_{X_1} +\alpha_{X_2} - \alpha_{\mcX} =  \rk \left ( \Pic_{X_1}(D_\mcX) \cap \Pic_{X_2}(D_\mcX) \right)$$
and
$$\alpha_{Y_1} +\alpha_{Y_2} - \alpha_{\mcY} =  \rk \left ( \Pic_{Y_1}(D_\mcY) \cap \Pic_{Y_2}(D_\mcY) \right),$$
 we give the following  definition.

\begin{definition} \label{mirrorII3}
Suppose that $d$-semistable  Calabi--Yau threefolds $\mathcal X = X_1 \cup X_2$,  $\mathcal Y = Y_1 \cup Y_2$ of type II satisfy
\begin{enumerate}
\item  $(X_i, Y_i)$ is a mirror pair of quasi-Fano threefolds  such that $D_\mcX=X_1\cap X_2$, $D_\mcY=Y_1\cap Y_2$  are anticanonical sections of $X_i$,  $Y_i$ respectively for each $i=1, 2$.
\item The pairs of lattices
$$\left(\Pic_{X_1}(D_\mcX) + \Pic_{X_2}(D_\mcX), \Pic_{Y_1}(D_\mcY) \cap \Pic_{Y_2}(D_\mcY))\right),$$
$$\left(\Pic_{X_1}(D_\mcX) \cap \Pic_{X_2}(D_\mcX)), \Pic_{Y_1}(D_\mcY) + \Pic_{Y_2}(D_\mcY) \right)$$
are $K3$-mirror pairs.
\end{enumerate}
Then the pair ($\mathcal X, \mathcal Y$) is called a mirror pair of  $d$-semistable Calabi--Yau threefolds of type II.
\end{definition}

Note that the pair $(X_{\Delta} \cup_D X_{\Delta}, Y_{\Delta^\circ} \cup_D Y_{\Delta^\circ})$, constructed in Sections \ref{sec2-2} and \ref{sec3-2}, is a mirror pair of $d$-semistable Calabi--Yau threefolds of type II in this definition.

\begin{theorem} \label{3dimmirrort}
For normal crossing varieties $\mathcal X = X_1 \cup X_2$,  $\mathcal Y = Y_1 \cup Y_2$, if  ($\mathcal X, \mathcal Y$) is a mirror pair of  $d$-semistable Calabi--Yau threefolds of type II, then the Calabi--Yau threefolds $M_{\mathcal X}, M_{\mathcal Y}$ satisfy
$$h^{1,1}(M_{\mathcal X}) = h^{1,2}(M_{\mathcal Y}), h^{1,2}(M_{\mathcal X}) = h^{1,1}(M_{\mathcal Y}).$$
\end{theorem}

\begin{proof}
\begin{align*}
 h^{1, 1} (M_\mcX) &= h^{1,1} (X_1) + h^{1,1} (X_2) - \alpha_\mcX -1\\
 &= h^{1,1} (X_1) + h^{1,1} (X_2) -\alpha_{X_1} -\alpha_{X_2}  - \alpha_{\mcY} +20 -1\\
 &= -h^{1,2} (Y_1) - h^{1,2} (Y_2)  - \alpha_{\mcY} +21\\
 &=h^{1, 2} (M_\mcY).
\end{align*}
Similarly we can get the second equation.
\end{proof}

The following immediate corollary of this theorem is very useful for generating many examples of mirror pairs of  $d$-semistable Calabi--Yau threefolds of type II.
\begin{corollary}
Let $\mathcal X = X_1 \cup X_2$,  $\mathcal Y = Y_1 \cup Y_2$ be $d$-semistable  Calabi--Yau threefolds of type II such that
\begin{enumerate}
\item   $(X_i, Y_i)$ is a mirror pair of quasi-Fano threefolds   for each $i=1, 2$,
\item
$\Pic_{X_1}(D_\mcX)$,  $\Pic_{Y_2}(D_\mcY)$ are sublattices of    $\Pic_{X_2}(D_\mcX)$, $\Pic_{Y_1}(D_\mcY)$ respectively.
\end{enumerate}
Then the pair ($\mathcal X, \mathcal Y$) is a mirror pair of  $d$-semistable Calabi--Yau threefolds of type II.
\end{corollary}

For a fixed reflexive 3-polytope $\Delta$, consider $X_\Delta$, $Y_\Delta$.
From their construction, we can choose $D_{X_\Delta}$, $D_{Y_\Delta}$ so that $D_{X_\Delta} \simeq D_{Y_\Delta}$ (to be denoted by $D_\Delta$). Hence  we can make a $d$- semistable normal crossing variety
$$\mcZ_\Delta = X_\Delta \cup_D Y_\Delta$$ \label{zdelta}
by gluing along  $D_{\Delta}$.

To ensure the smoothability of
$\mcZ_\Delta $ to a Calabi--Yau threefold, we need to show that it is projective, i.e.\
we need to find some ample divisors $H_{X_\Delta}, H_{Y_\Delta}$ of $X_\Delta$, $Y_\Delta$ respectively such that
$H_{X_\Delta}|_{D_\Delta}, H_{Y_\Delta}|_{D_\Delta}$ are linearly equivalent.
Note that
$\pi_X:X_\Delta  \ra X(\Sigma_\Delta)$ is a blow-up along a smooth curve $c \in | -K_{X(\Sigma_\Delta)}|$  and
$\pi_Y:Y_\Delta  \ra X(\Sigma_\Delta)$ is a sequential blow-up along a smooth curve $c_1,c_2, \cdots, c_k$ such that
$$\sum_{v \in \partial \Delta \cap \mbz^3} D_v \cap D_{X({\Sigma_{\Delta}})}=c_1+ c_2+ \cdots + c_k.$$
Let $\Delta[0] =\{v_1, v_2, \cdots, v_l\}$   and  $\gamma_i = D_{X(\Sigma_\Delta)} \cap D_{v_i}$.
Let $\{v'_1, \cdots, v'_m \}$ be the set of all the integral points that lie on the relative interiors of some edges of $\Delta$, then
$$\ D_{X(\Sigma_\Delta)} \cap D_{v'_i} = \epsilon_{i1} + \cdots + \epsilon_{i a_i},$$
 where $\epsilon_{i1}, \cdots, \epsilon_{i a_i}$ are disjoint smooth  rational curves and $a_i = l^*(\theta^\circ)$ for $v'_i \in \theta$.

Then
$$\{c_1,c_2, \cdots, c_k \} = \{\gamma_1, \gamma_2, \cdots, \gamma_l   \} \cup \bigcup_{1\leq i \leq m} \{\epsilon_{i1}, \cdots, \epsilon_{i a_i} \}.$$
Let $E_i$, $F_{ij}$ be the exceptional divisors over $\gamma_i, \epsilon_{ij}$ in the sequential blow-up $\pi_Y:Y_\Delta  \ra X(\Sigma_\Delta)$ respectively.
For an ample divisor $H$ of $X(\Sigma_\Delta)$, there are some positive integers $b_i$'s, $d_{ij}$'s such that
the divisor
$$H_1:= N \pi_Y^*(H) - \sum_i b_i E_i - \sum_{i,j}d_{ij}F_{ij} $$
is ample on $Y_\Delta$ for sufficiently large $N$.
The point here is that we can assume that $d_{i1}=\cdots = d_{i a_i}$ ($= d_i$) since  the curve $\epsilon_{i1}, \cdots, \epsilon_{i a_i}$ are all disjoint.
Let $E$ be the exceptional divisor of the blow-up $\pi_X:X_\Delta  \ra X(\Sigma_\Delta)$.
Then the divisor
$$H_2:=N \pi_X^*(H) - \sum_i (b_i -1)\pi_X^*(D_{v_i}) - \sum_i (d_i-1) \pi_X^*(D_{v'_i})-E$$
is ample on $X_\Delta$ for sufficiently large $N$.
It is trivial to check that $H_{X_\Delta}|_{D_\Delta}, H_{Y_\Delta}|_{D_\Delta}$ are linearly equivalent.
So we proved that $\mcZ_\Delta = X_\Delta \cup_D Y_\Delta$ is projective.
Similarly $\mcZ_{\Delta^\circ} =Y_{\Delta^\circ} \cup_D X_{\Delta^\circ}$  is also $d$-semistable and projective.

Note that
$$(X_\Delta, Y_{\Delta^\circ} ), (Y_\Delta, X_{\Delta^\circ})$$
are mirror pairs of quasi-Fano threefolds and
\cen{$\Pic_{X_\Delta}(D_{X_\Delta}) \subset \Pic_{{Y_\Delta}}(D_{Y_\Delta})$, $\Pic_{{X_{\Delta^\circ}}}(D_{X_{\Delta^\circ}}) \subset \Pic_{{Y_{\Delta^\circ}}}(D_{Y_{\Delta^\circ}})$.}
Therefore $(\mcZ_\Delta, \mcZ_{\Delta^\circ})$ is a mirror pair of   $d$-semistable  Calabi--Yau threefolds of type II.
If $\Delta$ is self-dual, i.e.\ the polytopes $\Delta, \Delta^\circ$   are equivalent, then
$\mcZ_\Delta = \mcZ_{\Delta^\circ}$. So $\mcZ_\Delta$ (also $M_{\mcZ_\Delta}$) is a self-mirror.
There are 79 self-dual 3-polytopes among 4319 equivalence classes of reflexive 3-polytopes.
Hence this construction gives us 2199 mirror pairs of Calabi--Yau threefolds, including 79 self-mirrors.
Those are listed in Table 2 of \cite{Lee2} and the following are their Hodge numbers.
\begin{align*}
h^{1,1}(M_{\mcZ_\Delta}) = h^{1,2}(M_{\mcZ_{\Delta^\circ}}) &= l(\Delta)
       +| \Delta[0]| +\sum_{\theta \in {\Delta}[1] } l^*(\theta)  + \sum_{ \Gamma \in {\Delta}[2]} l^*(\Gamma )  -4, \\
\end{align*}
\begin{align*}
h^{1,2}(M_{\mcZ_\Delta}) = h^{1,1}(M_{\mcZ_{\Delta^\circ}} )   &=  23 - l( \Delta ) + l(\Delta^\circ) + \sum_{v \in {\Delta}^{[0]}} l^*(v^\circ)   +\sum_{ \Gamma \in {\Delta}[2]} l^*(\Gamma ) -  \sum_{\theta \in {\Delta}[1] } l^* (\theta^\circ)l^*(\theta). \\
\end{align*}

The polytope $\Delta$ in Example \ref{exam1} gives rise to the pair $(M_{\mcZ_\Delta}, M_{\mcZ_{\Delta^\circ}})$ of Calabi--Yau threefolds with
$$h^{1,1}(M_{\mcZ_\Delta}) = h^{1,2}(M_{\mcZ_{\Delta^\circ}})  = 13, h^{1,2}(M_{\mcZ_\Delta}) = h^{1,1}(M_{\mcZ_{\Delta^\circ}} ) = 37.$$

Note that there are some  multiplicities of mirror pairs in the Batyrev mirror construction in \cite{Ba}, due to the fact that desingularizations of Calabi--Yau hypersurfaces in Gorenstein toric Fano fourfolds may not be unique. There also have been found  multiplicities to some of Borcea--Voisin  mirror pairs (\cite{ABS}).
There are similar multiplicities in the mirror construction of this paper.

In the construction of quasi-Fano threefolds $X_\Delta, Y_\Delta$ which are blow-ups of $X(\Sigma_\Delta$), there may be more than one choice of maximal projective triangulations of $\partial \Delta$.
This leads to a multiplicity of mirror pairs of quasi-Fano threefolds.
Furthermore recall that we built  $Y_\Delta$ by sequentially blowing up $X(\Sigma_\Delta)$ along curves $c_1, c_2, \cdots, c_k$.
The sequential blow-up depends on the order of blow-ups in general (see comments after Remark 8.3 in \cite{Lee}) and so this also gives another multiplicity of mirror pairs.
These cause the multiplicities in the construction of mirror pairs, $\left(\Xi_{X_\Delta}, \Xi_{Y_{\Delta^\circ}} \right)$,
$\left( M_{\mcZ_\Delta}, M_{\mcZ_{\Delta^\circ}} \right)$.

\section{Are they new?} \label{sec6}

Our examples in Table 1, 2 of \cite{Lee2} are constructed by smoothing method while  Calabi--Yau threefolds in \cite{AlGrHe2, KrSk2}, which are the largest source of known examples of Calabi--Yau threefolds,  come as desingularizations of anticanonical sections of Gorenstein toric Fano fourfolds.
A simple way of distinguishing our examples from them is to compare the Hodge numbers but most of their Hodge numbers overlap.
That's probably because there are so many such examples from toric Fano fourfolds.
However one still could suspect that  the mirror pairs in this paper may overlap with  those from toric fourfolds.
In this section, we pick up a particular Calabi--Yau threefold from our list and  explicitly show that this  Calabi--Yau threefold is not homeomorphic to any of those from toric fourfolds.
The Calabi--Yau threefold which we pick up is $\Xi_{X_\Delta}$, where $X(\Sigma_\Delta) = \mbp^3$.
Its Hodge numbers are
$$h^{1,1} = 2, h^{1,2} = 86.$$
There are exactly ten  different Calabi--Yau threefolds with these Hodge numbers that are desingularizations of anticanonical sections of Gorenstein toric Fano fourfolds (\cite{Ba, KrSk2}).
Those are constructed from nine reflexive 4-polytopes  --- one of the polytopes gives rise to two non-homeomorphic Calabi--Yau threefolds.

For a compact threefold $M$ with $h^2(M)=2$ and the second Chern class $c_2(M)$  that is not zero in $H^4(M, \mbz)_f$, where $A_f = A/A_t$  for an Abelian group $A$ with its torsion part $A_t$, we will define a topological invariant $\lambda$ as follows.
Note that the subgroup
 $$ \{ l \in H^2(M, \mbz)_f | c_2(M) \cdot l =0 \} $$
of $H^2(M, \mbz)_f$ is generated by a single element $m$. Then the number
$$\lambda(M) := |m^3|$$
is a topological invariant of $M$.
Firstly we calculate $\lambda(\Xi_{X_\Delta})$. $\Xi_{X_\Delta}$ is a smoothing of normal crossing variety $\mathcal X = X_1 \cup X_2$, where
$$\pi_i:X_i \ra X(\Sigma_\Delta)$$
is a blow-up of $X(\Sigma_\Delta) = \mbp^3$ along a smooth curve $c \in |-K_{\mbp^3}|_{S}|$ and $D=X_1 \cap X_2$ is the proper transform in $X_i$ of $S$ for $i=1,2$, where $S$ is a smooth quartic surface  of $\mbp^3$.
Let
$$G^k(\mathcal X, \mbz) = \ker(H^k(X_1, \mbz) \oplus H^k(X_2, \mbz)   \ra H^k(D, \mbz )),$$
where the map $H^k(X_1, \mbz) \oplus H^k(X_2, \mbz)   \ra H^k(D, \mbz )$ is given by
$$(l_1, l_2) \mapsto l_1|_D - l_2|_D.$$
Note that $G^k(\mathcal X, \mbz)$ inherits the cup product from those of \mbox{$H^k(X_1, \mbz)$, $H^k(X_2, \mbz)$} with the mixed term set to be zero (see \S4,  \cite{Lee}).
Let
\cen{$h_1 = (\pi_1^*(H), \pi_2^*(H)), h_2 = (4\pi_1^*(H)-E_1,0)$,}
 where $H$ is a hyperplane section of $\mbp^3$ and $E_i$ be the exceptional divisor of the blow-up $\pi_i$.
Then it is easy to check that $h_1$, $h_2$ belong to $G^2(\mathcal X, \mbz)$.
Note
$$\{\pi_1^*(H),  4\pi_1^*(H)-E_1\}$$
 is a basis for the lattice $H^2(X_1, \mbz)_f$. By Poincar\'e duality, there are classes $l_1, l_2 \in H^4(X_1, \mbz)$ such that the cup product matrix of
 \cen{$\{\pi_1^*(H),  4\pi_1^*(H)-E_1\}$ and $\{l_1, l_2 \}$}
  is the $2 \times 2$ identity matrix.
Let
$$h'_1 = (l_1, 0), h'_2=(l_2, f),$$
 where $f$ is a fiber over a point on the blow-up center  under $\pi_2$.
It is not hard to see that $h'_1$, $h'_2$ belong to $G^4(\mathcal X, \mbz)$.
Now the cup product matrix of $\{h_1, h_2\}$ and $\{h'_1, h'_2\}$ is the $2 \times 2$ identity matrix, which is unimodular.
According to \cite{Lee}, this property guarantees that there is an isomorphism $\phi$ from the sublattice $\langle h_1, h_2 \rangle$ of $G^2(\mathcal X, \mbz)$ to $H^2(M_{\mathcal X}, \mbz)_f$ with the cup product preserved. Let $c_2 = (c_2(X_1), c_2(X_2))$, then $c_2$ belongs to $G^4(\mathcal X, \mbz)$ (\S7, \cite{Lee}) and
$$c_2 \cdot h = c_2( M_{\mathcal X}) \cdot \phi(h)$$
for any $h \in \langle h_1, h_2 \rangle $.
Now we can calculate $\lambda(M_{\mathcal X})$.
Note
$$c_2 \cdot h_1 = 44, c_2 \cdot h_2 = 24,$$
So the group
$$\{ h \in \langle h_1, h_2 \rangle  |c_2 \cdot h=0\}$$
is generated by $6h_1 -11 h_2$. Hence
$$\lambda(\Xi_{X_\Delta})=\lambda(M_{\mathcal X}) = |(6h_1 -11 h_2)^3| =4320.$$

Next we need to find out the  $\lambda$-invariants of those ten Calabi--Yau threefolds.
Since those threefolds are hypersurfaces in toric varieties, it is a routine job to calculate the cubic forms on the second integral cohomology groups and the product with the second Chern class. Those calculations are provided in  a data base (\cite{AlGrHe, AlGrHe2}). Using this data base, we calculate  $\lambda$-invariants of those ten Calabi--Yau threefolds in Table \ref{lambda}, where `ID \#' is the polytope number in \cite {AlGrHe2}.
One can find out the vertex coordinates of the  corresponding 4-polytopes in \cite{AlGrHe2} with those polytope ID \#'s.

\begin{table}[h]
\begin{center}
\caption{$\lambda$-invariants of Calabi--Yau threefolds with
$$h^{1,1}=2, h^{1,2} =86$$
 in Gorenstein toric Fano fourfolds } \label{lambda}
\begin{tabular}{|c|c|c|c|c|c|c|c|c|c|c|}
\hline
ID \# & 12&13&13&14&15&16&17&18&19&20 \\
\hline
$\lambda(M)$& 1404&108&1564&3456&17280&17946&137214&67230&258198&457050 \\
\hline
\end{tabular}
\end{center}
\end{table}

Since all the $\lambda$-invariants are different from
$\lambda(\Xi_{X_\Delta}) =4320$, the Calabi--Yau threefold $\Xi_{X_\Delta}$ is not homeomorphic to a desingularization of an anticanonical section of any  Gorenstein toric Fano fourfold.

\section{Higher dimensional cases } \label{sec7}


For higher dimensional cases,  we give a definition, taking  the mirror relation (\ref{hypermirr}) for anticanonical sections  into account.
\begin{definition} \label{higherdefqm}
A pair $(X, Y)$ of quasi-Fano manifolds of dimension higher than three that have anticanonical Calabi--Yau fibrations
$$\overline W_X : X \ra \mbp^1, \overline W_Y : Y \ra \mbp^1,$$
 is called a mirror pair if
the pairs
$$(X^*, Y), (Y^*, X)$$
satisfy (\ref{koneq1}), (\ref{koneq2}) respectively and
$(D_X, D_Y)$ is a mirror pair of Calabi--Yau manifolds, where $D_X, D_Y$ are generic smooth anticanonical sections of $X, Y$ respectively.
\end{definition}
We also generalize Definition \ref{mirrorII3} for higher dimensions.

\begin{definition}
Let $\mathcal X = X_1 \cup X_2$,  $\mathcal Y = Y_1 \cup Y_2$  be $d$-semistable  Calabi--Yau manifolds of type II with dimension  higher than three.
If  $(X_i, Y_i)$ is a mirror pair of quasi-Fano manifolds such that $D_\mcX=X_1\cup X_2$, $D_\mcY=Y_1\cup Y_2$  are  anticanonical sections of $X_i, Y_i$ respectively   for each $i=1, 2$.
Then the pair ($\mathcal X, \mathcal Y$) is called a mirror pair of  $d$-semistable Calabi--Yau manifolds of type II.
\end{definition}
We have a higher dimensional apology of Theorem \ref{3dimmirrort} (see also Theorem 2.3, \cite{DoHaTh}).
\begin{proposition}For normal crossing varieties $\mathcal X = X_1 \cup X_2$,  $\mathcal Y = Y_1 \cup Y_2$ of dimension $n$, if  ($\mathcal X, \mathcal Y$) is a mirror pair of $d$-semistable Calabi--Yau manifolds of type II, then
$$\chi(M_\mcX) = (-1)^{n} \chi(M_{\mathcal Y}).$$
\end{proposition}
\begin{proof}
Firstly
\begin{align*}
\chi({X_i}) =& \chi({X_i}^* ) + \chi(D_{X_i} )\\
        =&  \chi({X_i}^*, W_{X_i}^{-1}(t)) + \chi(W_{X_i}^{-1}(t)) +  \chi(D_{X_i} )\\
        =& (-1)^n \chi({Y_i}) +  2 \chi(D_{X_i} )\\
\end{align*}
and similarly we have
\begin{align*}
\chi({Y_i}) =& (-1)^n \chi({X_i}) +  2 \chi(D_{Y_i} ).\\
\end{align*}
By (\ref{hypermirr}),
$$\chi(D_{X_i} ) = \chi(D_{\mcX}) = (-1)^{n-1}\chi(D_{\mcY}) = (-1)^{n-1} \chi(D_{Y_i} ).$$
Hence we have
\begin{align*}
\chi(M_\mcX) &= \chi(X_1) + \chi(X_2) - 2\chi(D_\mcX)\\
             &= (-1)^n \chi({Y_1}) + (-1)^n \chi({Y_2}) +2\chi(D_\mcX) \\
             &= (-1)^n \chi({Y_1}) + (-1)^n \chi({Y_2}) + 2(-1)^{n-1}\chi(D_\mcY)\,\,\,\,\,\,\,\,\, \left(\because \,\,(\ref{hypermirr})\right) \\
              &= (-1)^n \left( \chi({Y_1}) +  \chi({Y_2}) - 2\chi(D_\mcY) \right)\\
              &= (-1)^n \chi(M_\mcY).
\end{align*}
\end{proof}

If we try to construct  $\Xi_{X_\Delta}$, $\Xi_{Y_{\Delta}}$ and $M_{\mcZ_{\Delta}}$ for higher dimensional reflexive polytope $\Delta$, we come up with some difficulties due to singularities as follows.
\begin{enumerate}
\item There may be no smooth  maximal partial projective crepant desingularization  $X(\Sigma_\Delta)$ of $\mbp(\Delta)$ for $\dim \geq 4$.
\item  There may be  no smooth anticanonical section of $X(\Sigma_\Delta)$ for $\dim \geq 5$.
\end{enumerate}
Hence we may not apply the smoothing theorem in \cite{KaNa} for these cases.
However one can build $\Xi_X$ as a double cover of $X$, branched along $D_X \cup D_X'$, where $D_X'$ is another anticanonical section of $X$, disjoint from $D_X$.
So it is still possible to construct $\Xi_{X_\Delta}$, $\Xi_{Y_{\Delta}}$ and these will be some singular Calabi--Yau varieties, which are expected to satisfy the properties similar to those in Theorem \ref{dcmirror}.
In the case of $\mcZ_{\Delta}$, which may not be a normal crossing variety anymore, one needs to generalize the smoothing theorem so that some mild singularities may be allowed.
It seems natural to allow some mild singularities when one considers higher dimensional quasi-Fano manifolds.

\section{`Rigid' quasi-Fano manifolds} \label{sec8}

In this paper, we have considered a special kind of varieties --- quasi-Fano manifolds with anticanonical fibrations and defined notions of mirror pairs of them.
Noting that they are of negative Kodaira dimension and have additional fibration structure, their classifications seem reachable at least for the three-dimensional case with very low or very high $\alpha_X$.

There are Calabi--Yau threefolds that do not have Calabi--Yau threefolds as their mirror partners such as rigid ones ($h^{1,2} = 0$).
Hence it would be worthwhile to ask if quasi-Fano manifolds with anticanonical fibrations  always come as mirror pairs.
For three-dimensional case, if ($X, Y$) is a mirror pair of quasi-Fano threefolds with anticanonical fibrations, then we should have
$\alpha_X + \alpha_Y = 20$.
So if $\alpha_X$ has its maximal value, $20$, then  $\alpha_Y$ needs to be zero, which is impossible because $Y$ has an ample divisor.
Hence if $\alpha_X=20$, $X$ does not have a quasi-Fano threefold with anticanonical fibration as its mirror partner --- in this case, $X$ could be called `rigid' quasi-Fano threefold.
Such examples can be made easily. For example, take an exceptional ($\rk \Pic =20$) quartic $K3$ surface $D$ on $\mbp^3$ that has smooth curves $c_1, c_2, \cdots, c_k$ such that
\begin{itemize}
\item $c_1 + c_2 + \cdots + c_k$ belongs to the linear system $|-K_{\mbp^3}|_D|$,
\item $c_1, c_2 \cdots, c_k$ generate $H^2(D, \mbq)$.
\end{itemize}
An example of such $D$ is the Fermat quartic.
Blow up sequentially $\mbp^3$ along $c_1, c_2, \cdots, c_k$ to get a quasi-Fano threefold $X$.
Then $\alpha_X = 20$ and so $X$ is a `rigid' quasi-Fano threefold which does not have a quasi-Fano threefold with anticanonical fibration as its mirror partner.

In the case of dimension $n>3$, there are also `rigid' quasi-Fano manifolds whose generic anticanonical sections are rigid Calabi--Yau manifolds. Take a rigid Calabi--Yau manifold $D$ of dimension $n-1$ with non-Gorenstein involution $\rho$ on it, where we call an involution $\rho$ non-Gorenstein if $\rho^*(\omega) = - \omega$ for each $\omega \in H^{n-1,0}(D)$.
We assume further that the fixed locus of $\rho$ is a manifold of dimension $n-2$.
Let $X$ be the blow-up of the quotient $(D \times \mbp^1) / (\rho, \tau)$ along the singular locus, where $\tau$ is an involution of $\mbp^1$, fixing two distinct points.
Then $X$ is a quasi-Fano manifold whose anticanonical section $D_X$ is isomorphic to $D$, which is a rigid Calabi--Yau manifold.
Hence in the view of Definition \ref{higherdefqm}, $X$ does not have a quasi-Fano manifold with anticanonical fibration as its mirror partner.
For $n=4$, an easy example of such a Calabi--Yau threefold $D$ is the one that was introduced  by Beauville  in \cite{Be}.

Besides those `rigid' quasi-Fano manifolds, there are `non-rigid' quasi-Fano threefolds that do not have
quasi-Fano threefolds as its mirrors.
Consider a quasi-Fano threefold $X$ such that the lattice
\begin{align}
L^\bot \cap H^2(D_X, \mbz) \label{duke}
\end{align}
does not contain a hyperbolic lattice, where $L = \Pic_X(D_X)$. Then $L$ does not have a $K3$-mirror lattice. So $X$ cannot have a quasi-Fano threefold as its mirror.
Some concrete examples are obtained from non-symplectic involutions  on $K3$ surfaces.
Choose a non-symplectic involution  $\rho$ on a $K3$ surface whose invariant lattice does not a $K3$-mirror lattice, then  the quasi-Fano threefold $V_\rho$ introduced in Section \ref{sec4} has  no  quasi-Fano threefolds as its mirror.
According to the classification in \cite{Ni}, there are 11 families of such involutions and the resulted quasi-Fano threefold $V_\rho$ satisfies
$$11 \le \alpha_{V_\rho} \le 19.$$
In sum, there  are `non-rigid' quasi-Fano threefolds that do not have
quasi-Fano threefolds as its mirrors.

If the lattice in (\ref{duke}) contains a hyperbolic lattice, then  $\Pic_X(D_X)$ has a $K3$-mirror lattice. Hence, an interesting question would be:

\begin{question}
\emph{For a quasi-Fano threefold $X$ with anticanonical fibration such that $\Pic_X(D_X)$ has a $K3$-mirror lattice, is there a quasi-Fano threefold that has $X$ as its mirror partner?}
\end{question}

This work was done during a visit to the University of Nebraska at Lincoln. The author is very grateful to his friend Kyungyong Lee and his family for their warm hospitality.
This work was supported by  Basic Science Research Program
through the National Research Foundation of Korea(NRF) funded by the Ministry of Education (NRF-2017R1D1A2B03029525).



\begin{thebibliography}{9999}

\bibitem{AlGrHe} Altman, Ross; Gray, James; He, Yang-Hui; Jejjala, Vishnu; Nelson, Brent D.
\textit{A Calabi--Yau database: threefolds constructed from the Kreuzer-Skarke list.}
J. High Energy Phys. 2015, no. 2, 158.

\bibitem{AlGrHe2} Altman, Ross; Gray, James; He, Yang-Hui; Jejjala, Vishnu; Nelson, Brent D.
\textit{Toric Calabi--Yau Database.}
{\tt{http://www.rossealtman.com/}}


\bibitem{ABS} Artebani, Michela; Boissi\'ere, Samuel; Sarti, Alessandra
\textit{Borcea--Voisin Calabi--Yau threefolds and invertible potentials.}
Math. Nachr. 288 (2015), no.\ 14--15, 1581--1591.




\bibitem{AsGrMo} Aspinwall, Paul S.; Greene, Brian R.; Morrison, David R.
\textit{The monomial-divisor mirror map.}
Internat. Math. Res. Notices 1993, no. 12, 319--337.


\bibitem{Au} Auroux, Denis
\textit{Special Lagrangian fibrations, mirror symmetry and Calabi--Yau double covers.}
G\'eom\'etrie diff\'erentielle, physique math\'ematique, math\'ematiques et soci\'et\'e. I.
Ast\'erisque No. 321 (2008), 99--128.




\bibitem{Ba} Batyrev, Victor V.
\textit{Dual polyhedra and mirror symmetry for Calabi--Yau hypersurfaces in toric varieties.}
J. Algebraic Geom. 3 (1994), no. 3, 493--535.


\bibitem{BaBo} Batyrev, Victor V.; Borisov, Lev A.
\textit{Mirror duality and string-theoretic Hodge numbers.}
Invent. Math. 126 (1996), no. 1, 183--203.



\bibitem{Be} Beauville, Arnaud
\textit{Some remarks on Kähler manifolds with c1=0. }
Classification of algebraic and analytic manifolds (Katata, 1982), 1--26,
Progr. Math., 39, Birkhäuser Boston, Boston, MA, 1983.



\bibitem{Bo} Borcea, Ciprian
\textit{$K3$ surfaces with involution and mirror pairs of Calabi--Yau manifolds.}
 Mirror symmetry, II, 717--743,
AMS/IP Stud. Adv. Math., 1, Amer. Math. Soc., Providence, RI, 1997.

\bibitem{Bor} Borisov, Lev A.
\textit{ Towards the Mirror Symmetry for Calabi--Yau Complete Intersections
in Gorenstein Toric Fano Varieties.}
 University of Michigan, Preprint 1993, {\tt{alggeom/9310001}}

\bibitem{CaLySc} Candelas, P.; Lynker, M.; Schimmrigk, R.
\textit{Calabi--Yau manifolds in weighted $\mbp^4$.}
Nuclear Phys. B 341 (1990), no. 2, 383--402.


\bibitem{CoKa}Cox, David A.; Katz, Sheldon
\textit{Mirror symmetry and algebraic geometry.}
Mathematical Surveys and Monographs, 68. American Mathematical Society, Providence, RI, 1999.

\bibitem{Do}Dolgachev, I. V.
\textit{Mirror symmetry for lattice polarized $K3$ surfaces.}
Algebraic geometry, 4.
J. Math. Sci. 81 (1996), no. 3, 2599--2630.



\bibitem{DoHaTh} Doran, C. F.;  Harder, A.;  Thompson A.
\textit{ Mirror symmetry, Tyurin degenerations and fibrations on Calabi--Yau manifolds}, preprint, January 2016. To appear in String-Math 2015.




\bibitem{Fr}Friedman, Robert
\textit{Global smoothings of varieties with normal crossings.}
Ann. of Math. (2)  118  (1983),  no. 1, 75--114.



\bibitem{GeKaZe}Gel'fand, I. M.; Kapranov, M. M.; Zelevinsky, A. V.
\textit{Discriminants, resultants, and multidimensional determinants.
Mathematics: Theory \& Applications.}
 Birkhäuser Boston, Inc., Boston, MA, 1994.

\bibitem{Gi} Givental, Alexander B.,
\textit{Equivariant Gromov-Witten invariants.}
Internat.\ Math.\ Res.\ Notices 1996, no.\ 13, 613--663.


\bibitem{Ka} Kasprzyk, Alexander Mieczyslaw
\textit{Toric Fano varieties and convex polytopes.}
Thesis (Ph.D.)-University of Bath (United Kingdom). 2006.

\bibitem{KaKoPa} Katzarkov, Ludmil; Kontsevich, Maxim; Pantev, Tony
\textit{Bogomolov-Tian-Todorov theorems for Landau--Ginzburg models.}
J. Differential Geom. 105 (2017), no. 1, 55--117.

\bibitem{KaNa} Kawamata, Yujiro; Namikawa, Yoshinori
\textit{Logarithmic deformations of normal crossing varieties and smoothing of degenerate Calabi--Yau varieties.}
Invent. Math. 118 (1994), no. 3, 395--409.

\bibitem{KeTa} Kenji Hashimoto; Taro Sano
\textit{Examples of non-K"ahler Calabi--Yau 3-folds with arbitrarily large $b_2$.}
\text{arXiv:1902.01027}


\bibitem{KlSh} Klemm, Albrecht; Schimmrigk, Rolf
\textit{Landau--Ginzburg string vacua.}
Nuclear Phys. B 411 (1994), no. 2-3, 559--583.


\bibitem{Kn}Knutsen, Andreas Leopold
\textit{Smooth curves on projective K3 surfaces.}
Math. Scand. 90 (2002), no. 2, 215--231.

\bibitem{KoLe} Kovalev, Alexei; Lee, Nam-Hoon
\textit{$K3$ surfaces with non-symplectic involution and compact irreducible G2-manifolds.}
Math. Proc. Cambridge Philos. Soc. 151 (2011), no. 2, 193--218.



\bibitem{KrSk} Kreuzer, Maximilian; Skarke, Harald
\textit{Classification of reflexive polyhedra in three dimensions. }
Adv. Theor. Math. Phys. 2 (1998), no. 4, 853--871.

\bibitem{KrSk2} Kreuzer, Maximilian; Skarke, Harald
\textit{Complete classification of reflexive polyhedra in four dimensions.}
Adv. Theor. Math. Phys. 4 (2000), no. 6, 1209--1230.



\bibitem{Ku}Kulikov, Vik. S.
\textit{Degenerations of $K3$  surfaces and Enriques surfaces.}
Izv. Akad. Nauk SSSR Ser. Mat.  41  (1977), no. 5, 1008--1042.



\bibitem{Lee} Lee, Nam-Hoon
\textit{Calabi--Yau construction by smoothing normal crossing varieties.}
Internat. J. Math. 21 (2010), no. 6, 701--725.

\bibitem{Lee2} Lee, Nam-Hoon
\textit{6518 mirror pairs of Calabi--Yau threefolds:  Appendix to ``Mirror pairs of Calabi--Yau theefolds from mirror pairs of quasi-Fano threefolds''}
{\tt http://newton.kias.re.kr/{\textasciitilde}nhlee/files/appendix.pdf}


\bibitem{LLY} Lian, Bong H.; Liu, Kefeng; Yau, Shing-Tung,
\textit{Mirror principle. I.}
Asian J.\ Math.\ 1 (1997), no.\ 4, 729--763.


\bibitem{Ni} Nikulin, V. V.
\textit{Discrete reflection groups in Lobachevsky spaces and algebraic surfaces.}
 Proceedings of the International Congress of Mathematicians, Vol. 1, 2 (Berkeley, Calif., 1986), 654--671, Amer. Math. Soc., Providence, RI, 1987.


\bibitem{Ro}  Rohsiepe, Falk
\textit{Lattice polarized toric $K3$ surfaces}
{\tt arXiv:hep-th/0409290}



\bibitem{Ty} Tyurin, Andrei N.
\textit{Fano versus Calabi--Yau.}
  The Fano Conference, 701--734, Univ. Torino, Turin, 2004.


\bibitem{Vo} Voisin, Claire
\textit{Miroirs et involutions sur les surfaces $K3$}.
Journ\'ees de G\'eom\'etrie Alg\'ebrique d'Orsay (Orsay, 1992).
Ast\'erisque No. 218 (1993), 273--323.




\end{thebibliography}
\end{document}